\documentclass[12pt]{amsart}
\usepackage{amsmath,amsthm}
\usepackage{geometry}                
\usepackage{graphicx}
\usepackage{amssymb}
\usepackage{epstopdf}
 \usepackage{xcolor}
\newtheorem{theorem}{Theorem}[section]
\newtheorem{corollary}[theorem]{Corollary}
\newtheorem{lemma}[theorem]{Lemma}
\newtheorem{proposition}[theorem]{Proposition}

\theoremstyle{definition}
\newtheorem{definition}[theorem]{Definition}

\theoremstyle{remark}
\newtheorem{remark}[theorem]{Remark}

\numberwithin{equation}{section}

\begin{document}

\title[] 
 {Necessary and Sufficient Conditions for a Triangle Comparison Theorem}

\author{James J. Hebda }
\address{Department of Mathematics and Statistics\\ 
Saint Louis University, St. Louis, MO 63103}
\email{james.hebda@slu.edu} 

 
 \author{Yutaka Ikeda}

\address{8 Spring Time CT\\ St. Charles, MO 63303 }
\email{ yutaka.ikeda@gmail.com}




\maketitle

\begin{abstract}
We prove a version of  Topogonov's triangle comparison theorem with surfaces of revolution as model spaces.  Given a  model surface and a Riemannian manifold with a fixed base point, we give necessary and sufficient conditions under which every geodesic triangle in the manifold with a vertex at the base point has a corresponding Alexandrov triangle in the model.   Under these conditions we  also prove a version of the Maximal Radius Theorem and a  Grove--Shiohama type Sphere Theorem.
\end{abstract}

\section{Introduction}
 
Let $\widetilde M$ be a   simply connected,  complete, 2--dimensional Riemannain manifold which is rotationally symmetric about its base point $\tilde o$, and  let $M$ be a complete Riemannian manifold with a fixed base point $o$. The generalized Toponogov comparison theorem asserts that, under appropriate hypotheses,  geodesic triangles $\triangle opq$ in $M$ have  a corresponding  geodesic triangle $\triangle \tilde o\tilde p\tilde q$ in $\widetilde M$,  whose  corresponding sides have the same lengths and  whose  corresponding angles have smaller measures.  Different versions of this theorem have appeared in the literature under increasingly more general hypotheses. For  a sample of this literature,  see
   \cite{DE,UA,YI,MK,YM,YM2,IMS,T, KT, KT2,ST,ISU}.   Typical hypotheses include bounding the curvature of $M$  from below by that of $\widetilde M$ and  imposing additional restrictions either on  $\widetilde M$ or on the triangles under consideration. 
In any case,  the hypotheses assumed in these works are stronger than needed.  In \cite{HI}, we proved a generalized Toponogov Theorem in which   the usual hypothesis  on curvature was replaced by a weaker notion,  called weaker radial attraction.  However, in that paper $\widetilde M$ was required to have the special property that the cut locus of every point $\tilde p$  in $\widetilde M$ is contained in  the meridian opposite to $\tilde p$.
In this paper, we place no restriction on  $\widetilde M$, but instead require that the geodesics in the space $M$ do not have bad encounters with  the cut loci in $\widetilde M$.  As it turns out,  the  assumption  of both  weaker radial attraction and  no bad encounters  is a necessary and sufficient condition for
the existence of a comparison triangle $\triangle \tilde o \tilde p \tilde q$ in $\widetilde M$ for every geodesic triangle $\triangle opq$ in $M$.  The  condition of no bad encounters,  which is defined in Section 4, is in the spirit of, but not equivalent to, a condition in \cite{ISU} that serves a similar purpose.

Before stating the main result, we introduce some notation and terminology. 

\begin{definition} If $\triangle opq$ is a geodesic triangle in $M$,  $ \sigma$ will always denote the side joining $p$ to $q$, $\gamma$  the  side joining $o$ to $q$ and $\tau$ the side joining $o$ to $p$.  (See Figure 1.) The corresponding sides in a corresponding geodesic triangle $\triangle\tilde o \tilde p\tilde q$ in $\widetilde M$ will be denoted $\tilde\sigma$, $\tilde\gamma$ and $\tilde\tau$ respectively.
We will say that $\triangle \tilde o\tilde p\tilde q$ is an \emph{Alexandrov triangle} corresponding to 
 $\triangle opq$ if the following three properties are satisfied:
 \begin{enumerate}
\item  Equality of corresponding sides:
$$ d(o,p) = d(\tilde o,\tilde p), \quad
 d(o,q) = d(\tilde o,\tilde q), \quad
 d(p,q) = d(\tilde p,\tilde q).$$
 \item 
 Alexandrov convexity 
 from the base point:  $$ d(\tilde o,\tilde \sigma(t)) \leq d(o,\sigma(t))\enspace  \forall t \in [0, d(p,q)].$$
 \item  The angle comparisons:
$$\measuredangle \tilde p \leq \measuredangle p,\quad \measuredangle \tilde q \leq \measuredangle  q.$$
 \end{enumerate}
 Here $d$ is the distance function in $M$ and $\widetilde M$.
\end{definition}

\begin{remark}
If $ d(o,p)$ and $d(o,q)$ are both strictly less than $\ell =\sup_{\tilde q \in \widetilde M} d(\tilde o,\tilde q)$, then (1) and (2) automatically imply  (3) by \cite[Lemma 4.6]{HI}.  
\end{remark}

\begin{theorem}\label{t:main}
Let $(M,o)$ be a complete pointed Riemannian manifold, and let $(\widetilde M,\tilde o)$ be  a   simply connected,  complete, 2--dimensional Riemannain manifold which is rotationally symmetric about $\tilde o$.  Every geodesic triangle $\triangle opq$ in $M$ has a corresponding Alexandrov triangle $\triangle \tilde o\tilde p\tilde q$ in $\widetilde M$ 
if and only if $\widetilde M$ has weaker radial attraction than $M$ and no minimizing geodesic in $M$ has a bad encounter with the cut locus in $\widetilde M$.

Furthermore, under these equivalent conditions, in addition to properties (1), (2), and (3), the Alexandrov triangle $\triangle\tilde o\tilde p\tilde q$ corresponding to  $\triangle opq$ also satisfies:
\begin{enumerate}
\item[(4)] the angle comparison at the base:
$$\measuredangle \tilde o \leq \measuredangle  o,$$ and
\item[(5)] the convexity conditions:
$$ d(\tilde p,\tilde \gamma(s)) \leq d(p,\gamma(s))\enspace  \forall s \in [0, d(o,q)],$$
$$ d(\tilde q,\tilde \tau(s)) \leq d(q,\tau(s))\enspace  \forall s \in [0, d(o,p)].$$
 \end{enumerate}
\end{theorem}

\begin{remark}  (i)  The necessity of  the weaker radial attraction hypothesis for the existence of Alexandrov triangles  was proved in \cite[Proposition 4.13]{HI}. (ii) The angle comparison at the base (5) was observed in \cite{IMS}  when $\widetilde M$ is a Von Mangoldt surface that bounds the radial curvature of $M$ from below.
(iii) The convexity conditions (5) in Theorem \ref{t:main} seem not to have been noted  in the previous literature. (iv) Theorem \ref{t:main} generalizes the main theorem of \cite{HI} because 
the assumption made in \cite{HI} that the cut loci of points in $\widetilde M$ are contained in the opposite meridian automatically  entails the hypothesis  that there are no bad encounters.

\end{remark}

When $\widetilde M$ has weaker radial attraction than $M$ we prove an analog of the Rauch  Theorem that compares the lengths of Jacobi fields along radial geodesics in the two manifolds.   Consequently, if $\widetilde M$ is compact, then $M$ is compact and $\max_{q \in M} d(o,q)  \leq \ell = \max_{\tilde q \in \widetilde M} d(\tilde o,\tilde q)$ with equality holding if and only if the metric on $M$ takes a special form. This Maximal Radius Theorem generalizes a result in \cite{IMS1}.

\begin{theorem}[Maximal Radius Theorem]\label{t:MRT}
Suppose $(\widetilde M, \tilde o)$ is a compact model surface  with radius $\ell < \infty$, whose metric takes the form  
$$ dr^2 + y(r)^2 d\theta^2$$
in polar coordinates $(r,\theta)$ about $\tilde o$.
Suppose that every geodesic triangle $\triangle opq$ in $M$ has a corresponding  Alexandrov triangle $\triangle\tilde o\tilde p\tilde q$ in $\widetilde M$.
If there is a point $q$ in $M$  with $d(o,q) = \ell$, then $M$ is diffeomorphic to $S^n$ and its metric takes the form
$$ dr^2 + y(r)^2 d\theta_{n-1}^2$$
in geodesic coordinates about $o$ where $d\theta_{n-1}^2$ is the standard metric on the unit $(n-1)$ sphere.
\end{theorem}

This paper is organized as follows:
In Section 2 we discuss  the reference maps of $M$ and $\widetilde M$.  The reference maps have used in   \cite{ISU}.  In this section we also give a sufficient condition for the existence of an Alexandrov triangle corresponding to a given geodesic triangle in terms of the slope field in the reference space of $\widetilde M$.  In Section 3 we discuss the notion of weaker radial attraction and draw some consequences. In particular we prove an analog of the Rauch Theorem and deduce  Theorem \ref{t:MRT}. In Section 4 we introduce the notion of  geodesics in $M$ having bad encounters with the cut locus in $\widetilde M$, and  prove the necessity of the hypothesis of no bad encounters in Theorem \ref{t:main}. We also investigate conditions that prevent bad encounters. In Section 5 we prove sufficiency of the conditions in Theorem \ref{t:main}. The examples in  Section 6  illustrate the hypotheses of Theorem \ref{t:main}.    Section 7  provides two topological applications of the main theorem including a Grove--Shiohama type Sphere Theorem. In Section 8 we employ our methods  to establish  some of the results in \cite{ISU}. Finally in Section 9, we calculate the slope field in the reference space for $\widetilde M$ at cut points.

\section{The Reference Map}

Let $M$ be a complete Riemannian manifold with base point $o$.  
Fix a point   $p \in M$ and fix a minimizing geodesic $\tau$  from $o$ to $p$. For any $q\in M$ different from $o$ and $p$, we determine a geodesic triangle $\triangle opq$ by choosing a minimizing geodesic $\sigma$ joining $p$ to $q$ and a minimizing geodesic $\gamma$ joining $o$ to $q$.  The notation $\triangle opq$ can be ambiguous  when $q$ is in the cut locus of either $p$ or  $o$ since the triangle depends on the choices of $\sigma$ and $\gamma$, but  in what follows,   the context will always make clear what geodesics form the sides of the triangle $\triangle opq$.
The reference map $F: M \rightarrow \mathbf{R}^2$ is defined by $F(q) = (d(p,q),d(o,q))$.  Clearly $F$ depends upon the base point $o$ and on $p$.  Setting $ r_0 = d(o,p)$, the triangle inequality implies that the image of $M$ under $F$ is contained in a certain half infinite oblique strip in the plane, that is,
$$ F(M) \subset   \{(x,y) \in \mathbf{R}^2: x+y \geq r_0,  -r_0 \leq y -x \leq r_0 \}. $$

\begin{figure}[htbp]
\begin{center}
\setlength{\unitlength}{1 cm}
\begin{picture}(6,6)(-1,-1)
{\color{gray}}
\thicklines
\put(0,2){\line(1,-1){2}}
\put(-0.3,2){\small$p$}
\put(2,-.3){\small$o$}
\put(3.1,4){\small$q$}
\put(2,0){\line(1,4){1}}
\put(0,2){\line(3,2){3}}
\put(2.7,2){\small$\gamma$}
\put(1.2,3){\small$\sigma$}
\put(1,1.2){\small$\tau$}
\end{picture}
\begin{picture}(6,6)(-1,-1)
{\color{gray}
\put(0,0){\line(1,0){5}}
\put(0,0){\line(0,1){5}}}
\put(2.5,4.5){\line(1,-1){2}}
\put(.7,1.3){\line(1,1){2.5}}
\thicklines
\put(0,2){\line(1,-1){2}}
\put(0,2){\line(1,1){3}}
\put(2,0){\line(1,1){3}}
\put(-0.7,1.8){$^{ F(p)}$}
\put(1.8,-.5){$^{F(o)}$}
\put(3.3,3.6){$^{F(q)}$}
\qbezier(2,0)(3.2,2.9)(3.2,3.8)
\qbezier(0,2)(3,4)(3.2,3.8)
\put(0.4,0.4){$^{F(\tau)}$}
\put(3,2){$^{F(\gamma)}$}
\put(1,2.3){$^{F(\sigma)}$}
\end{picture}
\caption{$\triangle opq$ and $F(\triangle opq)$.}
\label{f:1}
\end{center}
\end{figure}
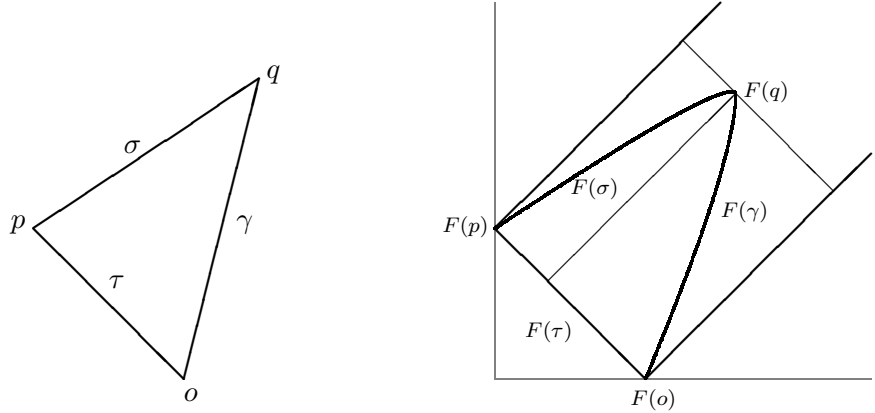

\begin{lemma} \label{l:rectangle}
Given $\triangle opq$, $F(\sigma)$ lies in the rectangle with diagonal  $ F(p) F(q)$, and $F(\gamma)$ lies in the rectangle  with diagonal $F(o) F(q)$. See Figure \ref{f:1}.
\end{lemma}

\begin{proof}
Let $q^\ast$ be a point on $\sigma$ between $p$ and $q$, and set 
$$ F(q^\ast) = (x,y) = (d(p,q^\ast),d(o,q^\ast)).$$
Since $\sigma$ is a minimizing geodesic, one has  $x + d(q^\ast,q) = d(p,q)$. 
Therefore, applying the triangle inequality one  obtains
\begin{eqnarray*}
d(o,p) &\leq& x + y  \cr
&\leq& x + d(q^\ast,q) + d(q,o) \cr
&=& d(p,q) + d(o,q)
\end{eqnarray*}
and 
\begin{eqnarray*}
d(o,q) -d(p,q) &=& d(o,q) - (d(q^\ast,q) + x) \cr
                     &\leq& y - x \cr
                          &\leq& d(o,p).
\end{eqnarray*}
\begin{figure}[htbp]
\begin{center}
\setlength{\unitlength}{1 cm}
\begin{picture}(6,6)(-1,-1)
{\color{gray}}
\put(0,2){\circle*{0.15}}
\put(-0.3,2){\small$p$}
\put(2,0){\circle*{0.15}}
\put(2,-.3){\small$o$}
\put(3,4){\circle*{0.15}}
\put(3.1,4){\small$q$}
\put(1.5,3){\circle*{0.15}}
\put(1.3,3.2){ $q^*$}
\put(0,2){\line(1,-1){2}}
\put(2,0){\line(1,4){1}}
\put(0,2){\line(3,2){3}}
\put(1.5,3){\line(1,-6){.5}}
\put(.58,2.6){\small$x$}
\put(1.8,1.8){\small$y$}
\end{picture}
\begin{picture}(6,6)(-1,-1)
{\color{gray}}
\put(0,2){\circle*{0.15}}
\put(-0.3,2){\small$p$}
\put(2,0){\circle*{0.15}}
\put(2,-.3){\small$o$}
\put(3,4){\circle*{0.15}}
\put(3.1,4){\small$q$}
\put(2.5,2){\circle*{0.15}}
\put(2.5,2){ $q^*$}
\put(0,2){\line(1,-1){2}}
\put(2,0){\line(1,4){1}}
\put(0,2){\line(3,2){3}}
\put(0,2){\line(1,0){2.5}}
\put(1.3,2.1){\small$x$}
\put(2.4,1){\small$y$}
\end{picture}
\caption{$\triangle opq$ with $q^*$ on $\sigma$ (left) and with $q^*$ on $\gamma$ (right).}
\label{f:1a}
\end{center}
\end{figure}
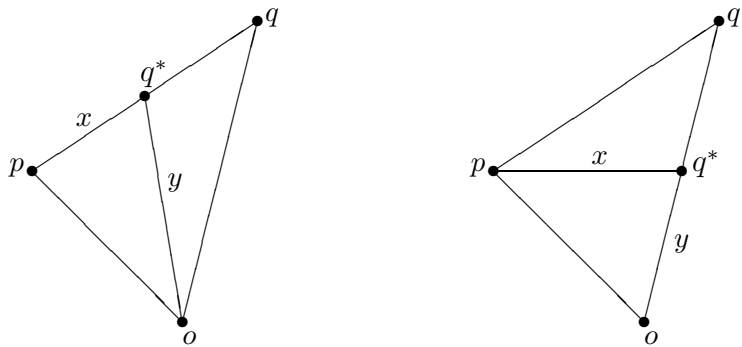
See Figure \ref{f:1a} on the left.
This proves the statement about $F(\sigma)$. If  $q^\ast$ is a point on $\gamma$ between $o$ and $q$ and $F(q^\ast) = (x,y)$, then the statement about $F(\gamma)$ is similarly proved  by showing that
\begin{equation*}
 d(o,p) \leq x + y \leq d(o,q) + d(p,q)
\end{equation*}
and
\begin{equation*}
-d(o,p) \leq y - x \leq d(o,q) - d(p,q)
\end{equation*}
using the minimizing property of $\gamma$ and the triangle inequality. See Figure \ref{f:1a} on the right.
\end{proof}
One should observe that the image $F(\sigma)$ is the graph of the function $L_o\circ\sigma$ where $L_o(q) = d(o,q)$.

\subsection{Model Surfaces}

Let $\widetilde M$ be a complete surface which is rotationally symmetric about the base point $\tilde o$, and let $(r,\theta)$ denote polar coordinates on $\widetilde M$.
Suppose that $r_0 < \ell$ where $\ell \in (0,\infty]$ denotes the supremum of the distance function from $\tilde o$.  We pick $\tilde p$ to be the point with polar coordinates $(r_0, 0)$.  We have the corresponding reference map $\widetilde F : \widetilde M \rightarrow \mathbf{R}^2$  defined by $\widetilde F(\tilde q) = (d(\tilde p,\tilde q),d(\tilde o,\tilde q))$. The meridian opposite $\tilde p$ is defined by $\theta = \pi$. Set 
$\widetilde M^+$ to be the portion of $\widetilde M$  satisfying $0 \leq \theta \leq \pi$.  Then $\widetilde F $ carries  $\widetilde M^+$ homeomorphically onto $\widetilde F(\widetilde M)$ \cite{ISU}.  The interior $int(\widetilde M^+)$ of $\widetilde M^+$ consists of the points satisfying $ 0< r<\ell$, and $0 < \theta < \pi$.

We collect here some known properties of the cut locus  $C(\tilde p)$ of $\tilde p$ in a surface of revolution. The structure of cut loci in surfaces is discussed in \cite{JH} and in \cite{SST}.
\begin{enumerate}
\item If $\widetilde M$ is compact, then $C(\tilde p)$ is a tree. This means that the minimal connected subset containing any given pair of points is homeomorphic to an interval whose endpoints are the given pair.
If $\widetilde M$ is not compact, the cut locus may be disconnected, but still each connected component is a tree.

\item  We will call the portion of $C(\tilde p)$ which is contained in the meridian opposite $\tilde p$ the trunk of $C(\tilde p)$.  A  maximal connected piece of  $C(\tilde p) \cap int(\widetilde M^+)$ will be called a positive branch of the cut locus.   When $\widetilde M$ is not compact, the branches may not be attached to the trunk.  It is possible that there are branches but the trunk is empty. In these cases we say that the branch attaches to the trunk  at infinity. 

\item If $\tilde q$ is a point on a positive branch of $C(\tilde p)$, then there will generally be at least two minimizing geodesics joining $\tilde p$ to $\tilde q$.  If $\tilde\sigma_1$ and $\tilde\sigma_2$ are two minimizing geodesics joining $\tilde p$ to $\tilde q$, we say  $\tilde\sigma_1$ is above $\tilde\sigma_2$ if $d(\tilde o,\tilde \sigma_1(t)) > d(\tilde o,\tilde \sigma_2(t))$ for all $ 0 < t < d(\tilde p,\tilde q)$.
Under the reference map the curve $\widetilde F(\tilde\sigma_1)$ lies above the curve $\widetilde F(\tilde\sigma_2)$.  There is always an uppermost $\tilde\sigma^\uparrow$ and a lowermost $\tilde\sigma^\downarrow$ minimizing geodesic joining $\tilde p$ to $\tilde q$.  However, $\tilde\sigma^\downarrow = \tilde\sigma^\uparrow$ in the case that $\tilde q$ is an endpoint of the branch of $C(\tilde p)$ and there is a unique minimizing geodesic joining $\tilde p$ to $\tilde q$.

\item If $\tilde q$ is a point on a positive branch of $C(\tilde p)$, then there is an arc $\alpha$ in the branch of $C(\tilde p)$ joining $\tilde q$ to the trunk. The points on $\alpha$ are parameterized by their distance from $\tilde p$. If $ a = d(\tilde p, \tilde q)$ and $b$ is the distance of the point where the branch attaches to the trunk, (which may be $\infty$), then $\alpha(a)=\tilde q$ and the  point $\alpha(t)$ satisfies $d(\tilde p, \alpha(t)) = t$ for all $ a \leq t < b$.  Moreover the right--hand derivative $(L_{\tilde o}\circ \alpha)_+^\prime (t)$ exists for all $t \in [a,b)$ where $L_{\tilde o}(-) = d(\tilde o, -)$.  A formula for $(L_{\tilde o}\circ \alpha)_+^\prime (t)$ is obtained in Lemma \ref{l:formula} below.
\end{enumerate}

We define a slope field $\frak{s} : \widetilde F(int (\widetilde M^+))\rightarrow \mathbf{R}$ as follows:
Let $(x,y) \in \widetilde F(\widetilde M)$.  Suppose that  $(x,y) = \widetilde F(\tilde q)$
for some  $\tilde q \in int(\widetilde M^+)$. If $\tilde q \notin C(\tilde p)$, then there exists a unique minimizing geodesic $\tilde \sigma$ emanating  from $\tilde p$ that passes through $\tilde q$.  
Thus $\tilde\sigma(x) = \tilde q$, and we
 set $\frak{s} (x,y) = (L_{\tilde o}\circ \tilde\sigma)_+^\prime(x)$.  If $\tilde q$ is on a branch of $C(\tilde p)$, then let $\alpha$ be the arc in the cut locus joining $\tilde q$ to the trunk of $C(\tilde p)$. Then $\alpha(x) = \tilde q$. Set  $\frak{s} (x,y) = (L_{\tilde o}\circ \alpha)_+^\prime(x)$.    The slope field $\frak{s}$ has discontinuities at points of $\widetilde F(C(\tilde p))$ but is smooth in the complement of   $\widetilde F(C(\tilde p))$.
 The integral curves of $\frak{s}$ away from the cut points are the images under $\widetilde F$ of the geodesics emanating from $\tilde p$.   

A formula for $\frak{s}$ is easily computed for a 2--sphere of constant curvature.

\begin{proposition}\label{p:kappasphere}
Let $\widetilde M_\kappa$ be the 2-sphere of constant curvature $\kappa$ with base point $\tilde o$, and
let $\tilde p \in \widetilde M$ with $d(\tilde p, \tilde o) = r_0$.  Then the reference space $\widetilde F(\widetilde M_\kappa)$ is the rectangle
$$ \left\{ (x,y) : r_0 \leq x+ y \leq \frac{2 \pi}{\sqrt\kappa} - r_0, - r_0 \leq y-x \leq r_0 \right\}, $$
  and the slope field has the formula
\begin{equation} \label{e:kssf}
\frak{s}(x,y) =   \frac{ \cos(\sqrt\kappa r_0)-\cos(\sqrt\kappa y)\cos(\sqrt\kappa x)}{\sin({\sqrt\kappa x})\sin(\sqrt\kappa y)}. \end{equation}
\end{proposition} 

\begin{proof}
In polar coordinates about $\tilde o$, the unit speed geodesic $(r(t), \theta(t))$ with initial conditions $r(0)= r_0$ and $r^\prime(0) = \dot r_0$ satisfies:
$$ \cos(\sqrt\kappa r(t)) = \cos(\sqrt\kappa t)\cos(\sqrt\kappa r_0) - \dot r_0 \sin(\sqrt\kappa t)\sin(\sqrt\kappa r_0).$$
Since $\widetilde F(r(t),\theta(t) ) = (t,r(t))= (x,y)$, the image of this geodesic in the reference space is the solution curve of the equation
\begin{equation}\label{e:gks1}
\cos(\sqrt\kappa y) = \cos(\sqrt\kappa x)\cos(\sqrt\kappa r_0) - \dot r_0 \sin(\sqrt\kappa x)\sin(\sqrt\kappa r_0).
\end{equation}
Differentiating equation (\ref{e:gks1}) implicitly  gives
\begin{equation}\label{e:gks2}
\frak{s}(x,y) = \frac{dy}{dx} =  \frac{ \sin(\sqrt\kappa x)\cos(\sqrt\kappa r_0) + \dot r_0 \cos(\sqrt\kappa x)\sin(\sqrt\kappa r_0)}{\sin(\sqrt\kappa y)}.
\end{equation}
To obtain (\ref{e:kssf}),
solve for $\dot r_0$ in (\ref{e:gks1}),  substitute the result into (\ref{e:gks2}) and simplify.
\end{proof}

Following equation  (\ref{e:kssf}), Figure \ref{f:2a}  presents qualitative pictures of the slope field  $\frak s$ for  the 2--sphere of constant curvature $\kappa$  for different values of $r_0$ .  
\begin{figure}[htbp]
\setlength{\unitlength}{1 cm}
\begin{picture}(6,6)(-1,-1)
{\color{gray}
\put(0,0){\line(1,0){5}}
\put(0,0){\line(0,1){5}}
\qbezier(0,2)(1.123764312,2)(1.561882156,1.561882156)
\qbezier(2,0)(2,1,123764312)(1.561882156,1.561882156)
\qbezier(5,3)(3.876235688,3)(3.438117844,3.438117844)
\qbezier(3,5)(3,3.876235688)(3.438117844,3.438117844)}
\thicklines
\put(0,2){\line(1,-1){2}}
\put(0,2){\line(1,1){3}}
\put(3,5){\line(1,-1){2}}
\put(2,0){\line(1,1){3}}
\put(-0.5,1.8){$r_0$}
\put(1.8,-.5){$r_0$}
\put(0.9,1.2){\footnotesize$\mathfrak{s}<0$}
\put(2.4,2.4){\footnotesize$\mathfrak{s}>0$}
\put(3.5,3.5){\footnotesize$\mathfrak{s}<0$}
\put(-.6,5){\footnotesize$\frac {\pi} {\sqrt{\kappa}}$}
\put(4.8,-.5){\footnotesize$\frac {\pi} {\sqrt{\kappa}}$}
\end{picture}
\begin{picture}(6,6)(-1,-1)
{\color{gray}
\put(0,0){\line(1,0){5}}
\put(0,0){\line(0,1){5}}
\put(0,2.5){\line(1,0){5}}
\put(2.5,0){\line(0,1){5}}}
\thicklines
\put(0,2.5){\line(1,-1){2.5}}
\put(0,2.5){\line(1,1){2.5}}
\put(2.5,5){\line(1,-1){2.5}}
\put(2.5,0){\line(1,1){2.5}}
\put(-0.5,2.5){$r_0$}
\put(2.5,-.5){$r_0$}
\put(1.5,1.5){\footnotesize$\mathfrak{s}<0$}
\put(1.5,3.3){\footnotesize$\mathfrak{s}>0$}
\put(2.7,3.3){\footnotesize$\mathfrak{s}<0$}
\put(2.7,1.5){\footnotesize$\mathfrak{s}>0$}
\put(-.6,5){\footnotesize$\frac {\pi} {\sqrt{\kappa}}$}
\put(4.8,-.5){\footnotesize$\frac {\pi} {\sqrt{\kappa}}$}
\end{picture}
\begin{picture}(6,6)(-1,-1)
{\color{gray}
\put(0,0){\line(1,0){5}}
\put(0,0){\line(0,1){5}}
\qbezier(0,3)(1.123764312,3)(1.561882156,3.438117844)
\qbezier(2,5)(2,3.876235688)(1.561882156,3.438117844)
\qbezier(3,0)(3,1.123764312)(3.438117844,1.561882156)
\qbezier(5,2)(3.876235688,2)(3.438117844,1.561882156)}
\thicklines
\put(0,3){\line(1,-1){3}}
\put(0,3){\line(1,1){2}}
\put(2,5){\line(1,-1){3}}
\put(3,0){\line(1,1){2}}
\put(-0.5,2.8){$r_0$}
\put(2.8,-.5){$r_0$}
\put(0.9,3.7){\footnotesize$\mathfrak{s}>0$}
\put(2.4,2.4){\footnotesize$\mathfrak{s}<0$}
\put(3.5,1.4){\footnotesize$\mathfrak{s}>0$}
\put(-.6,5){\footnotesize$\frac {\pi} {\sqrt{\kappa}}$}
\put(4.8,-.5){\footnotesize$\frac {\pi} {\sqrt{\kappa}}$}
\end{picture}
\caption{The sign of the slope field  $\mathfrak{s}$ and its  nullclines $\mathfrak{s}= 0$ in
 $\widetilde F(\widetilde M^+_\kappa) $ for $ 0 < r_0< \frac \pi {2\sqrt{\kappa}}$ (upper left), $r_0 =    \frac \pi {2\sqrt{\kappa}}$ (upper right), and $ \frac \pi {2\sqrt{\kappa}}< r_0 <  \frac \pi {\sqrt{\kappa}}$ (lower center). }
\label{f:2a}
\end{figure}
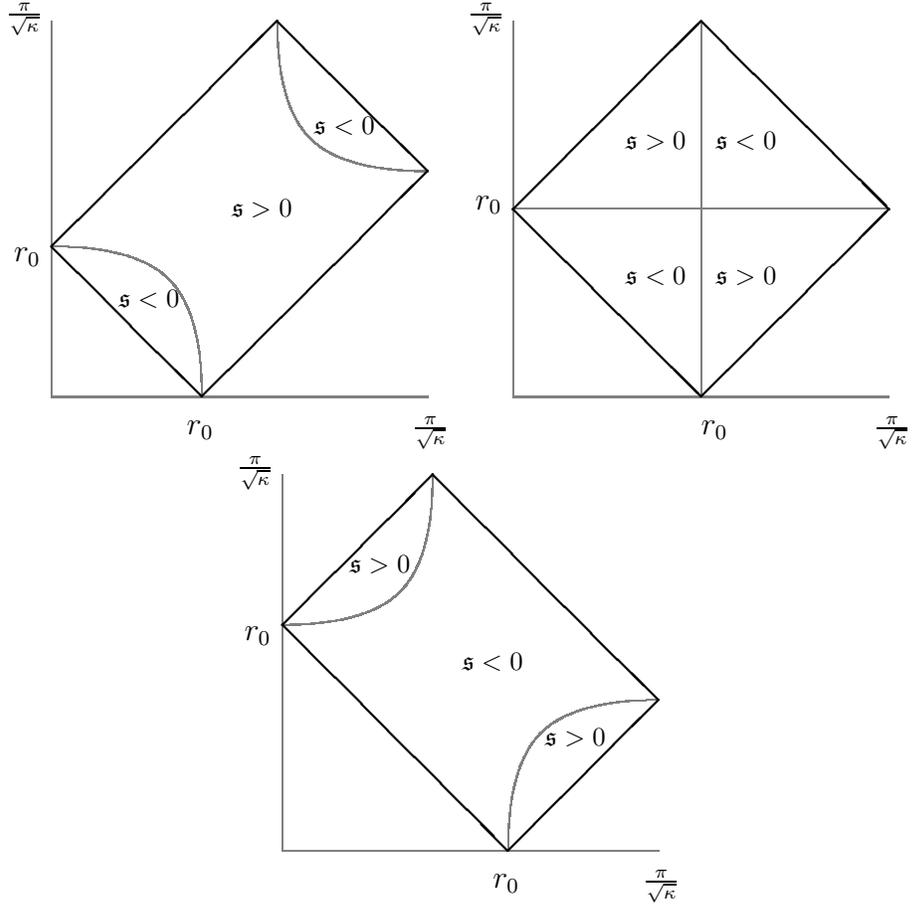

\begin{definition}\label{d:varsigmaphi}
 Let $\tilde\sigma_\phi$ be the geodesic emanating from $\tilde p$ making an angle $\phi$ with the meridian $\mu_0$ through $\tilde p$. Specifically, $\phi  \in [0,\pi]$ is the angle between  $\tilde \sigma^\prime_\phi(0)$ and $ - \mu_0^\prime(\tilde p)$.
 Suppose that $\tilde\sigma_\phi$ meets the cut locus $C(\tilde p)$ at parameter distance $\tau_\phi$.   If $\tilde\sigma_\phi(\tau_\phi)$ is a cut point in the trunk of $C(\tilde p)$, define $\varsigma_\phi (t) = \tilde\sigma_\phi(t)$ for $ 0 \leq t \leq \tau_\phi$, while if $\tilde\sigma_\phi(\tau_\phi)$ is on a branch of $C(\tilde p)$, let $\alpha$ be the arc in the cut locus joining $\tilde\sigma_\phi(\tau_\phi)$ to the trunk, and define $\varsigma_\phi = \tilde\sigma_{\phi} \cdot \alpha $, that is, the concatenation of $\tilde\sigma_\phi$ with $\alpha$ which is equal to $\tilde\sigma_\phi(t)$ for $0\leq t \leq \tau_\phi$ and to $\alpha(t)$ for $t \geq \tau_\phi$.
\end{definition}

 By construction  $L_{\tilde o}\circ \varsigma_\phi$ is a solution  of the slope field $\frak{s}$ in the sense that its right--hand derivative satisfies
$$ (L_{\tilde o}\circ \varsigma_\phi)^\prime_+ (t) = \frak{s}(\widetilde F\circ \varsigma_\phi(t)).$$

\begin{lemma}\label{l:ineq}
Suppose $f$, which is continuous and has a finite right--hand derivative for all $t$, satisfies the differential inequality
\begin{equation}\label{e:fplus}
 f^\prime_+(t) \leq \frak{s}(t, f(t)) 
 \end{equation}
 Fix $\phi_0 \in (0,\pi)$, and let $g(t) = L_{\tilde o}(\varsigma_{\phi_0}(t))$ so that $ g^\prime_+(t) = \frak{s} (t, g(t))$.  If for some $t_0 >0$, $f(t_0) \leq g(t_0)$, then $f(t) \leq g(t)$ for $t > t_0$. In other words, the graph of $f$ cannot cross the graph of $g$ from below to above.
\end{lemma}

\begin{proof}
Let $(x,y)$ be a point in $\widetilde F(\widetilde M^+)$ which is not in the image of $C(\tilde p)$. Define $\Phi(x,y)$ to be the angle $\phi$ which the minimizing geodesic joining $\tilde p$ to $\tilde q$ makes with the meridian from $\tilde p$ to $\tilde o$, where $\tilde q$ is the unique point in $\widetilde M^+$ such that $\widetilde F(\tilde q) = (x,y)$.  Thus $\Phi$ is $C^\infty$ in the complement of $\widetilde F(C(\tilde p))$.  By construction the level curves of $\Phi$ are the images under $\widetilde F$ of  minimizing geodesics emanating from $\tilde p$.  Thus at all points $(t, f(t))$ not in $\widetilde F(C(\tilde p))$, we have  that the right--hand derivative of $\Phi(t,f(t))$ is nonpositive because, by (\ref{e:fplus}),  $f_+^\prime(t)$ is less than or equal to the slope of the level curve of $\Phi$ passing through $(t,f(t))$.   Thus in those intervals where $(t,f(t))$  does not meet $\widetilde F(C(\tilde p))$, $\Phi(t,f(t))$ is nonincreasing. In particular, the graph of $f$ cannot cross any of the level curves of $\Phi$ from a lower to a higher value.

To prove the Lemma, suppose there exists a $ t _1> t_0$ with $f(t_1) > g(t_1)$. By continuity of $f$ and $g$ there exists a $\bar t$ with $t_0 \leq \bar t$ such the $g(\bar t) = f(\bar t)$ and $g(t) < f(t)$ for all $ \bar t < t < t_1$.  If $(\bar t, g(\bar t)) = \widetilde F(\varsigma_{\phi_0}(\bar t))$ is not in $\widetilde F(C(\tilde p))$, we would have by the previous paragraph that $\Phi(t,f(t))$ was nonincreasing in an interval about $\bar t$, and at the same time $\Phi(t,f(t)) > \Phi(\bar t, g(\bar t)) = \phi_0$ for $t > \bar t$.  Thus $(\bar t, g(\bar t)) \in \widetilde F(C(\tilde p))$. Pick a $t > \bar t$.
Then $\varsigma_{\phi_0}(t)= \widetilde F^{-1}(t,g(t))$  lies on the arc $\alpha$ in $C(\tilde p)$ starting at $\varsigma_{\phi_0}(\bar t)$. 
 Let $\tilde q \in C(\tilde p)$ be the point such that $\widetilde F (\tilde q) = ( t, g( t))$ and let $\tilde \sigma^\uparrow$ be the uppermost minimizing geodesic joining $\tilde p$ to $\tilde q$.  Thus the angle $\phi$ that $\tilde \sigma^\uparrow$ makes with the meridian through $\tilde p$ satisfies $\phi>\phi_0$. 
 Thus $L_{\tilde o}(\tilde \sigma^\uparrow(\bar t)) > g(\bar t) = f(\bar t)$ and $L_{\tilde o}(\tilde \sigma^\uparrow(t)) = g(t) < f(t)$.  Therefore the graph of $f$ must cross the image of $\tilde \sigma^\uparrow$ at some $ \bar t < s < t$, that is, it crosses a level curve of $\Phi$ from a lower to a higher value, contradicting  the earlier assertion.
 \end{proof}

Sufficient conditions under which a given geodesic triangle $\triangle opq$ has a corresponding Alexandrov triangle $\triangle\tilde o\tilde p\tilde q$ in $\widetilde M$  can be described in terms of the reference space and slope field $\frak{s}$  for $\widetilde M$.

\begin{proposition}
Given $\triangle opq$ in $M$, suppose $F(q) \in \widetilde F(\widetilde M)$.
Let $\sigma$ be the minimizing geodesic joining $p$ to $q$.  Every geodesic triangle $\triangle op\sigma(t)$ for $t \in (0, d(p,q)]$, has an Alexandrov triangle  $\triangle\tilde o\tilde p\tilde q_t$ in 
$\widetilde M$ if and only if  $\sigma$ satisfies the differential inequality.
$$( L_o\circ \sigma)^\prime_+(t) \leq \frak{s} (F(\sigma(t)))$$
for all $ 0 < t <d(p,q)$.
\end{proposition}

\begin{proof}
By Lemma \ref{l:rectangle}, $ F(\sigma) \subset \widetilde F(\widetilde M)$.  Thus every $\triangle op\sigma(t)$ has a corresponding  $\triangle\tilde o\tilde p\tilde q_t$ in 
$\widetilde M$ where we take the lowermost minimizing geodesic joining $\tilde p$ to $\tilde q_t$ for that side.

First assume that   for every  $t$   the  corresponding  $\triangle\tilde o\tilde p\tilde q_t$ in 
$\widetilde M$ is an Alexandrov triangle corresponding to $\triangle op\sigma(t)$.  Suppose the differential inequality is not satisfied for some $ t_0 \in (0,d(p,q))$.  Let $\phi_0$  be chosen so that  $\varsigma_{\phi_0}$ restricted to $[0,t_0]$ is the lowermost minimizing geodesic joining $\tilde p$ to $\tilde q_{t_0}$, that is, the side of the corresponding Alexandrov triangle $\triangle\tilde o\tilde p\tilde q_{t_0}$.  Since we are supposing that 
$$( L_o\circ \sigma)^\prime_+(t_0) > \frak{s} (F(\sigma(t_0))),$$
it follows that there exists $\epsilon > 0$ such that  
$L_o(\sigma(t)) > L_{\tilde o}(\varsigma_{\phi_0}(t))$ for all  $t \in ( t_0, t_0 + \epsilon)$.  
For any such $t$, let $\phi_t$ be chosen so that $\varsigma_{\phi_t}$ restricted to $[0,t]$ is the
lowermost minimizing geodesic joining $\tilde p$ to $\tilde q_t$. Thus for such $t$,  $L_{\tilde o}(\varsigma_{\phi_t}(t)) = L_o(\sigma(t)) > L_{\tilde o}(\varsigma_{\phi_0}(t))$. Hence we have  $\phi_t > \phi_0$, and therefore  $L_{\tilde o}(\varsigma_{\phi_t}(t_0))  > L_{\tilde o}(\varsigma_{\phi_0}(t_0)) =  L_o(\sigma(t_0))$ which contradicts Alexandrov convexity for $\triangle op\sigma(t)$
and $\triangle\tilde o\tilde p\tilde q_t$. Therefore the differential inequality is satisfied for all $t$.

Conversely assume that the differential inequality is satisfied for all $ 0 < t <d(p,q)$.  
We must show that for any $t$, $L_o(\sigma(s)) \geq L_{\tilde o}(\varsigma_{\phi_t}(s))$ for all $ 0 \leq s \leq t$.  If this were not so, there would exist a $\bar t$ and a $t_0$ with $ 0< t_0 < \bar t$ such that
$$ L_{\tilde o}(\varsigma_{\phi_{t_0}}(t_0)) = L_o(\sigma(t_0)) < L_{\tilde o}(\varsigma_{\phi_{\bar t}}(t_0)).$$ 
Hence $\phi_{t_0} < \phi_{\bar t}$ and therefore $  L_{\tilde o}(\varsigma_{\phi_{t_0}}(\bar t))  < L_{\tilde o}(\varsigma_{\phi_{\bar t}}(\bar t))= L_o(\sigma(\bar t))$.  
On the other hand, applying Lemma \ref{l:ineq} with  $f(t) = L_o(\sigma(t))$ leads to  the contradiction
 $L_o(\sigma(\bar t)) \leq  L_{\tilde o}(\varsigma_{\phi_{t_0}}(\bar t)) $. 

\end{proof}

\section{Weaker Radial Attraction}

\subsection{Definition and Equivalences}

We introduced the notion of weaker radial attraction in \cite{HI}.  One may be regard it as an assumption comparing small hinges. In this section we investigate several consequences of this condition.

\begin{definition}
The  model surface $(\widetilde M, \tilde o)$ is said to have  {\it weaker radial attraction} than  the pointed complete Riemannian manifold $(M,o)$, if,
for any unit speed geodesics $\sigma, \tilde \sigma$ in $M, \widetilde M$ respectively satisfying
$L_o\circ\sigma (0) =  L_{\tilde o}\circ\tilde\sigma (0)< \ell = \sup_{\tilde p \in \widetilde M} d(\tilde o,\tilde p)$ and $(L_o\circ\sigma)^\prime_+ (0) =  (L_{\tilde o}\circ\tilde\sigma)^\prime_+ (0)$, then there exists an $\epsilon > 0$ such that $L_o\circ\sigma (t) \leq  L_{\tilde o}\circ\tilde\sigma (t)$ for all $ 0 \leq t < \epsilon$.   Here $L_o$ and $L_{\tilde o}$ are the distance functions from $o$ and $\tilde o$ respectively.
\end{definition}

\begin{remark}
(i)  As pointed out in   \cite[Remark 4.2]{HI}, if the radial curvature of $M$ is  bounded from below by the curvature of  $\widetilde M$, then $\widetilde M$ has weaker radial attraction than $M$ but not conversely.  (ii)  The necessity of  weaker radial attraction  in Theorem \ref{t:main} was proved in \cite[Proposition 4.13]{HI}.
\end{remark}

The following theorem proved in \cite{HI} asserts that the condition of weaker radial attraction is equivalent to two other conditions.

\begin{theorem}[Theorem 5.3 \cite{HI}] \label{t:equiv}
The following are equivalent:
\begin{enumerate}
\item The Hessian of $L_{\tilde o}$ dominates the Hessian of $L_o$.
\item The principal curvatures of the geodesic spheres about $o$ are bounded from below by the curvature of the geodesic circles about $\tilde o$ of the same radius. 
\item $\widetilde M$ has weaker radial attraction than $M$.
\end{enumerate}
\end{theorem}

It will be convenient to reformulate condition (1) in terms of a certain tensor field $S$ in $M$.
Suppose the metric of the model surface  $\widetilde M$ takes the form
$$ d\tilde s^2 = dr^2 + y(r)^2d\theta^2,$$
in polar cordinates about $\tilde o$.
Let $g$ denote the Riemannian metric for $M$. We  will also  write $\langle-,-\rangle = g(-,-)$. In the open set $M\backslash(C(o) \cup\{o\})$ we can define the radial vector field $\xi = grad(L_o)$ and the symmetric $(2,0)$ tensor field
$$  S = \frac{y^\prime\circ L_o}{y\circ L_o} \left(g - dL_o\otimes dL_o\right) - \nabla^2L_o.$$
For vector fields $X$ and $Y$, we have
$$S(X,Y) =  \frac{y^\prime\circ L_o}{y\circ L_o}( \langle X,Y\rangle - \langle X,\xi\rangle\langle Y,\xi\rangle) - \langle \nabla_X\xi, Y\rangle . $$
The metrically equivalent symmetric operator $\widehat S$, that is, the $(1,1)$ tensor field, defined by $$\langle \widehat S(X),Y\rangle = S(X,Y),$$  is thus given by
$$ \widehat S(X) =    \frac{y^\prime\circ L_o}{y\circ L_o}(  X - \langle X,\xi\rangle\xi) -  \nabla_X\xi.$$
It is clear from these formulas that $\widehat S(\xi) =0$.

\begin{corollary}\label{c:equiv}
The model surface 
 $(\widetilde M,\tilde o)$  has weaker radial attraction than  $(M,o)$ if and only if
 $S(X,X) \geq 0$ for all $X$, or equivalently, the eigenvalues of $\widehat S$ are nonnegative at every point $p \in M\backslash (C(o)\cup \{o\})$ with $L_o(p) < \ell$. 
\end{corollary}
\begin{proof}
Since $\nabla^2L_o$ is the Hessian of $L_o$ and $\frac {y^\prime}{y}( d\tilde s^2 - dr^2)$ is the Hessian of $L_{\tilde o}$ \cite[Proposition 2.20]{GW}, $S(X,X) \geq 0$ is exactly the condition that the Hessian of $L_{\tilde o}$ dominates the Hessian of $L_o$. The result follows from Theorem \ref{t:equiv}.
\end{proof}  

\begin{remark}
We will prove in Corollary  \ref{c:rauch1} that $L_o(p) \leq \ell$ for all $p \in M$. Thus $S$ is indeed defined in $M\backslash (C(o)\cup \{o\})$ and not just in $\{p:L_o(p) < \ell \} \cap M\backslash (C(o)\cup \{o\})$.
\end{remark}

\subsection{Geodesic comparison}

\begin{proposition}\label{p:geodcomp}
Suppose that $(\widetilde M, \tilde o)$ has weaker radial attraction than $(M,o)$. Suppose that $\sigma: [a,b] \rightarrow M$ and $\tilde\sigma: [a,b]\rightarrow \widetilde M$ are unit speed geodesics such that $ L_{\tilde o}\circ \tilde\sigma(t) \leq L_o\circ \sigma(t)$ for all $t \in [a,b]$.  If any one of the following three conditions hold:
\begin{enumerate}
\item $(L_o\circ \sigma)(a) = (L_{\tilde o}\circ \tilde\sigma)(a)$ and $(L_o\circ \sigma)^\prime_+(a) = (L_{\tilde o}\circ \tilde\sigma)^\prime_+(a)$.
\item $(L_o\circ \sigma)(b) = (L_{\tilde o}\circ \tilde\sigma)(b)$ and $(L_o\circ \sigma)^\prime_-(b) = (L_{\tilde o}\circ \tilde\sigma)^\prime_-(b)$.
\item There exists $t_0 \in (a,b)$ with $(L_o\circ \sigma)(t_0) = (L_{\tilde o}\circ \tilde\sigma)(t_0)$.
\end{enumerate}
then  $ (L_o\circ \sigma)(t) = (L_{\tilde o}\circ \tilde\sigma)(t)$ for all $t \in [a,b]$.
\end{proposition}

\begin{proof} 
If (1) holds, then by hypothesis and by weaker radial attraction there exists an $\epsilon > 0$ such that
$$ 
 (L_{\tilde o}\circ \tilde\sigma)(t) \leq ( L_o\circ \sigma)(t) \leq (L_{\tilde o}\circ \tilde\sigma)(t)
$$
for all $ a \leq t < a+\epsilon$.  Hence  $(L_{\tilde o}\circ \tilde\sigma)(t) = ( L_o\circ \sigma)(t)$ for  $ a \leq t < a+\epsilon$ and thus (3) holds.  Similarly if (2) holds so does (3).  We may now assume that (3) holds.
  Let $E = \{ t \in (a,b) : ( L_o\circ \sigma)(t) = (L_{\tilde o}\circ \tilde\sigma)(t)\}$. By assumption $E$ is nonempty. By continuity $E$ is closed.  It is also open because if $t_0 \in E$, then, using \cite[Corollary 2.3]{HI} for the middle inequality,
 $$ (L_{\tilde o}\circ \tilde\sigma)^\prime(t_0 ) \geq (L_o\circ \sigma)^\prime_-(t_0) \geq (L_o\circ \sigma)^\prime_+(t_0) \geq  (L_{\tilde o}\circ \tilde\sigma)^\prime(t_0 ).$$
Thus we have equality holding at all places.  By weaker radial attraction, there exists an 
$\epsilon > 0$ such that  $( L_o\circ \sigma)(t) \leq (L_{\tilde o}\circ \tilde\sigma)(t)$ for all $t \in (t_0-\epsilon, t_0+\epsilon)$. Combining this with the assumption,  $( L_o\circ \sigma)(t) = (L_{\tilde o}\circ \tilde\sigma)(t)$ for all $t \in (t_0-\epsilon, t_0+\epsilon)$. This shows that $E$ is open.  Thus by connectivity, $E = (a,b)$.  By continuity this equality extends to the endpoints as well.
 \end{proof}

\begin{lemma} 
Assume that $(\widetilde M, \tilde o)$ has weaker radial attraction than $(M,o)$. Let $q \in M\backslash( C(o)\cup\{o\}) $ with $L_o(q) = r < \ell$,  and let $X \in T_qM$ be linearly independent of the radial vector $\xi_q$ at $q$. If $\widehat S(X) =0$ then the sectional curvature $K(X \wedge \xi) = -\frac {y^{\prime\prime}(r)}{y(r)} $.  This is the Gaussan curvature of $\widetilde M$ at distance $r$ from $\tilde o$.
\end{lemma}
\begin{proof}
We may assume $X$ is a unit vector perpendicular to $\xi$ at $q$.  Since $\widehat S(X) = 0$, it follows that $ \nabla_X\xi = \frac {y^\prime(r)}{y(r)}X$.
  Let $\gamma(t)$ for $0\leq t \leq r$ be the minimizing geodesic joining $o$ to $q$, and let $J(t)$ be the Jacobi field along $\gamma$ satisfying $J(0)=0$ and $J(r)=X$.  The function $f(t) = S(J(t),J(t)) \geq 0$ for all  $0 < t < r+\epsilon$ for some positive $\epsilon$ and $ f(r) = 0$.  Hence $f$ attains its minimum at $t=r$. Therefore  
after a straightforward calculation, which uses $ \nabla_X\xi = \frac {y^\prime(r)}{y(r)}X$, we obtain
$$ 0 = f^\prime(r) = \frac {y^{\prime\prime}(r)}{y(r)}   + \langle R(X,\xi)\xi,X\rangle. $$ 
\end{proof}

Suppose that $\sigma$ is a geodesic in $M\backslash (C(o)\cup\{o\})$ and $\tilde \sigma$ a geodesic in $\widetilde M$ parameterized on the same interval such that $L_o\circ \sigma = L_{\tilde o}\circ \tilde \sigma$. Because 
$$\nabla^2L_o(\sigma^\prime, \sigma^\prime ) = (L_o\circ \sigma)^{\prime\prime} = (L_{\tilde o}\circ \tilde\sigma)^{\prime\prime} = \nabla^2L_{\tilde o}(\tilde\sigma^\prime, \tilde\sigma^\prime ),$$
it  follows that $\widehat S(\sigma^\prime) =0$ at all points along $\sigma$. The preceding lemma implies that for all $t$, $K(\sigma^\prime(t)\wedge \xi)$ equals the curvature of $\widetilde M$ at $\tilde\sigma(t)$.

\begin{proposition} \label{p:xiperp}
Assume the geodesic $\sigma$ in $M\backslash (C(o)\cup\{o\})$ satisfies $\widehat S(\sigma^\prime) =0$.  Define the vector field $\xi^\perp$ along $\sigma$ by 
$$ \xi^\perp = \xi - \langle \xi, \sigma^\prime\rangle \sigma^\prime.$$
Then
$$ \nabla _{\sigma^\prime} \xi^\perp =   - \frac {y^\prime\circ L_o\circ \sigma}{y\circ L_o\circ\sigma}\langle \sigma^\prime, \xi\rangle \xi^\perp. $$
\end{proposition}
\begin{proof}
This is a straightforward calculation using  $\nabla_{\sigma^\prime}\sigma^\prime = 0$ and
$$ \nabla_{\sigma^\prime} \xi = \frac {y^\prime\circ L_o\circ \sigma}{y\circ L_o\circ\sigma}(\sigma^\prime - \langle \sigma^\prime,\xi\rangle\xi
)  $$
which holds because $\widehat S(\sigma^\prime) =0$.
\end{proof}    
    
 \begin{corollary}
 Under the same hypothesis, the normalized vector field $\frac {\xi^\perp}{|\xi^\perp|}$ is parallel along $\sigma$, that is, 
 $$ \nabla_{\sigma^\prime} \frac {\xi^\perp}{|\xi^\perp|} =0$$
 because the covariant derivative of $\xi^\perp$ along $\sigma$ is a multiple of itself.
 \end{corollary}   
    
   This shows that the 2--planes spanned by $\sigma^\prime$ and $\xi$ are parallel along $\sigma$. 
    
    \begin{corollary}\label{c:conjugate}
    Suppose the two geodesics $\sigma: [0,a] \to M\backslash (C(o)\cup \{o\})$  and  $\tilde\sigma : [0,a] \to \widetilde M$ satisfy
    $L_o\circ \sigma = L_{\tilde o}\circ \tilde \sigma$, and suppose that $\tilde\sigma(a)$ is the first conjugate point to $\tilde\sigma(0)$ along $\tilde\sigma$. Then $\sigma$ is not free of conjugate points, and thus cannot minimize past $a$.
    \end{corollary}

     \begin{proof}
     We assume that $\sigma$ is free of conjugate points.  By the  Morse Index Lemma \cite[Corollary 3.2, p. 74]{KN}, if $V$ is a vector field along $\sigma$ which is perpendicular to $\sigma$ and satisfies $V(0)=0$ and $V(a)=0$, then 
   $$ I(V) = \int_0^a \langle \nabla_{\sigma^\prime} V, \nabla_{\sigma^\prime} V\rangle - \langle R(V,\sigma^\prime)\sigma^\prime,V\rangle dt \geq 0$$
 with equality holding if and only if $V$ is identically zero.  We will construct a non--zero vector field $V$ whose Morse index is $0$ to obtain a contradiction.
     
 Let $\kappa(r) = -\frac {y^{\prime\prime}(r)}{y(r)}$. Since   $\tilde\sigma(a)$ is the first conjugate point to $\tilde\sigma(0)$ along $\tilde\sigma$, there exists a non--zero Jacobi field along $\tilde\sigma$ which vanishes at $0$ and $a$.  Thus there is a non-zero function $f$ vanishing at $0$ and $a$ satisfying
 $$ f^{\prime\prime} + (\kappa\circ L_{\tilde o} \circ \tilde\sigma )f  = 0.$$
    Set $V = f \frac {\xi^\perp}{|\xi^\perp|}$. Then since $\frac {\xi^\perp}{|\xi^\perp|}$ is parallel along $\sigma$ we have, using Proposition \ref{p:xiperp} and integration by parts,
   \begin{eqnarray*}
I(V)&=& \int_0^a (f^\prime(t))^2 - (\kappa\circ L_{\tilde o} \circ \tilde\sigma )f (t)^2 dt \\
&=&   f^\prime(t)f(t) \vert_0^a - \int_0^a ( f^{\prime\prime}(t) +(\kappa\circ L_{\tilde o} \circ \tilde\sigma (t))f (t))f(t) dt \\
&=& 0.
    \end{eqnarray*}
     \end{proof}

The proof of this corollary shows that more generally:
    
     \begin{proposition} Suppose that $\gamma$ is a geodesic in a Riemannian manifold and  that $P$ is a parallel unit vector field along $\gamma$.  Let $\kappa(t) = K(\gamma^\prime\wedge P(t))$ for $0\leq t \leq a$.  If there exists a nonzero solution $f(t)$ of $$ f^{\prime\prime} + \kappa f = 0, $$ such that $f(0) = f(a) = 0$, then $\gamma$ is not free of conjugate points on $[0,a]$.
    \end{proposition}

\begin{remark} \label{r:conjugate}   
In  light of this proposition,  the conclusion of Corollary \ref{c:conjugate} is still valid if either one or both of the endpoints of $\sigma$ lie in $C(o)$, as long as the interior of $\sigma$ is disjoint from $C(o)\cup\{o\}$. This is because by continuity the parallel field  $\frac {\xi^\perp}{|\xi^\perp|}$ along the interior of $\sigma$  extends to a parallel field $P$ that satisfies
 $K(\sigma^\prime \wedge P)= \kappa\circ L_{\tilde o} \circ \tilde \sigma$ on the closed interval $[0,a]$.
\end{remark}

\subsection{Jacobi Fields}

We prove an analog of the Rauch Comparison Theorem for Jacobi fields   along a pair of geodesics in two Riemannian manifolds under an assumption on the Hessians of the distance functions, rather than under the usual assumption on the sectional  curvatures.

 \begin{theorem}\label{t:rauch}
 Let $M$ and $\bar M$ be two Riemannian manifolds.  Let $\gamma : [0,a] \rightarrow M$ and 
 $\bar\gamma : [0,a] \rightarrow \bar M$ be normal geodesics in $M$ and $\bar M$ respectively. Set $o = \gamma(0)$ and $\bar o = \bar\gamma(0)$.  Suppose that $\gamma(t)$ and $\bar\gamma(t)$ are not conjugate to $o$ and $\bar o$ respectively along $\gamma$ and $\bar\gamma$ respectively, and that the Hessians of $L_o$ and $L_{\bar o}$ satisfy $\nabla^2 L_o \leq \nabla^2L_{\bar o}$ at $\gamma(t)$ and $\bar\gamma(t)$ for every $0<t<a$.  If $J$ and $\bar J$ are Jacobi fields along $\gamma$ and $\bar \gamma$ respectively, that satisfy $J(0) =0$, $\bar J(0) = 0$, $| J^\prime(0)| = | \bar J^\prime (0)|$, and $\langle J^\prime(0) , \gamma^\prime(0) \rangle = \langle \bar J^\prime(0) , \bar\gamma^\prime(0) \rangle$, then
 \begin{equation*}
| J(t) | \leq |\bar J (t) |
\end{equation*}
for all $ 0\leq t \leq a$.  Moreover, if equality holds for some $t_0 \in (0,a)$, then equality holds for all $0 \leq t \leq t_0$, and if 
$$ \lim_{t \to a^-} \frac {| \bar J(t)|}{|J(t)|} = 1, $$
then equality holds for all $ 0 \leq t \leq a$.
 \end{theorem}
 
 \begin{proof}
 We may suppose that $\langle J^\prime(0) , \gamma^\prime(0) \rangle = \langle \bar J^\prime(0) , \bar\gamma^\prime(0) \rangle= 0$ so that $J$ and $\bar J$ are perpendicular to $\gamma$ and $\bar \gamma$ respectively.  Let $f(t) = \langle J(t), J(t) \rangle$ and $\bar f (t) = \langle \bar J(t), \bar J(t) \rangle$.  We must show $f(t) \leq \bar f(t)$. By two applications of l'H\^opital's rule,
 \begin{equation}\label{e:lhospital}
\lim_{t\rightarrow 0^+} \frac {\bar f(t)}{f(t)}  = \lim_{t\rightarrow 0^+} \frac {\langle \bar J^\prime(t), \bar J(t)\rangle}{\langle J^\prime(t), J(t)\rangle} =
 \lim_{t\rightarrow 0^+} \frac {\langle \bar J^\prime(t), \bar J^\prime(t)\rangle +   \langle \bar J^{\prime\prime} (t), \bar J(t)\rangle } {\langle  J^\prime(t),  J^\prime(t)\rangle +   \langle J^{\prime\prime} (t), J(t)\rangle } = 1.
\end{equation}
It thus suffices to prove
\begin{equation}\label{e:derivative}
 \frac {d}{dt}\left(  \frac{\bar f(t)} {f(t)} \right) \geq 0
\end{equation}
or equivalently that
\begin{equation*}
\frac {f^\prime(t)}{f(t)} \leq \frac{\bar f^\prime(t)}{\bar f(t)}
\end{equation*}
for all $ t \in (0,a)$. (The no conjugacy condition implies that $f$ and $\bar f$ do not vanish anywhere in $(0,a)$.)  Fix $t_0 \in (0,a)$.  Set
\begin{equation*}
Y(t) = \frac {1} {\sqrt {f(t_0)} } J(t) \quad\mathrm{and}\quad \bar Y(t) = \frac {1} {\sqrt {\bar f(t_0)}} \bar J(t)
\end{equation*}
so that $Y$ and $\bar Y$ are Jacobi fields along $\gamma$ and $\bar\gamma$ which are perpendicular to the geodesics and have the same norm at $t_0$, that is, $|Y(t_0)| = |\bar Y(t_0)| = 1 $.
Therefore, by the assumption on the Hessians,
\begin{equation*}
\frac {f^\prime(t_0)}{f(t_0)} =2\nabla^2L_o(Y(t_0), Y(t_0)) \leq 2 \nabla^2L_{\bar o}(\bar Y(t_0), \bar Y(t_0)) 
=\frac {\bar f^\prime(t_0)}{\bar f(t_0)}. 
\end{equation*}
This proves the inequality.  In case equality holds at $t_0 \in (0,a)$, equations (\ref{e:lhospital}) and (\ref{e:derivative}) imply that
\begin{equation*}
  \frac{\bar f(t)} {f(t)} = 1
\end{equation*}
for all $ t \in (0,t_0)$ and thus $f(t) = \bar f(t)$ whenever $0 \leq t \leq t_0$.  Similarly  $f(t) = \bar f(t)$ for all $0 \leq t \leq a$ if $\lim_{t \to a^-} \frac {\bar f(t)}{f(t)} = 1$.
 \end{proof}

\begin{corollary}  \label{c:rauch1}
Assume the model surface $(\widetilde M, \tilde o)$ has weaker radial attraction than the complete pointed Riemannian manifold $(M,o)$.  If $\widetilde M$ has a finite radius $\ell$, then  $d (o,p) \leq \ell$ for all $p \in M$.
\end{corollary}
\begin{proof} 
If not, there exists
a geodesic $\gamma$ emanating from $o$ which is conjugate free on the interval $[0,\ell]$.  
Let $\tilde \gamma$  be a geodesic emanating from $\tilde o$.  Then  $\gamma(t)$ and $\tilde \gamma(t)$ are not conjugate to $o$ and $\tilde o$ respectively for all $ t \in (0,\ell)$.  Because of weaker radial attraction,  $\nabla^2 L_o \leq \nabla^2L_{\tilde o}$ at $\gamma(t)$ and $\tilde\gamma(t)$ respectively for every $0<t<\ell$.  Let $\tilde J$  be a nontrivial Jacobi field along $\tilde  \gamma$ satisfying $\tilde J(0)=0$, $|\tilde J^\prime(0)| = 1$, and $\langle \tilde\gamma^\prime(0), \tilde J^\prime(0) \rangle = 0$.  Then $\tilde J(\ell) = 0$. Let $J$ be a Jacobi field along $\gamma$  satisfying
$ J(0)=0$, $| J^\prime(0)| = 1$, and $\langle \gamma^\prime(0),  J^\prime(0) \rangle = 0$.  By Theorem \ref{t:rauch},
$$ | J(\ell) | \leq |\tilde J(\ell) | = 0 $$ 
which contradicts that $\gamma$ is conjugate free for $t \in [0,\ell]$.
\end{proof}

More can be said in the case of equality in Theorem \ref{t:rauch}.

\begin{proposition}\label{p:eqlty}
Assume the hypothesis of Theorem \ref{t:rauch}.  Suppose $J$ and $\bar J$ are Jacobi fields along $\gamma$ and $\bar \gamma$ respectively which are perpendicular to the geodesics.  Assume 
 \begin{equation*}
| J(t) | = |\bar J (t) | = y(t)
\end{equation*} 
for all $ 0\leq t \leq t_0$.  Then there exist parallel unit vector fields $P$ and $\bar P$ along $\gamma$ and $\bar \gamma$ such that
 \begin{equation*}
J(t) =  y(t)P \quad\mathrm{and}\quad  \bar J(t) = y(t) \bar P.
\end{equation*}
\end{proposition}

\begin{proof}
From the proof of Theorem \ref{t:rauch} we have
\begin{equation*}
\nabla^2L_o(J(t), J(t)) = \nabla^2L_{\bar o}(\bar J(t), \bar J(t)) 
\end{equation*}
for all $t \in (0,t_0)$.   Since $\nabla^2 L_o \leq \nabla^2L_{\bar o}$ at $\gamma(t)$ and $\bar\gamma(t)$ for every $0<t<a$, this shows that $J(t)$ is the eigenvector for the largest eigenvalue of $\nabla^2L_o$ while $\bar J(t)$ is that for the smallest eigenvalue of $\nabla^2L_{\bar o}$ and those eigenvalues are equal for $t \in (0,t_0)$.   This common eigenvalue can be denoted $\lambda(t)$  and is a continuous function of $t$.  
What this means is that
\begin{equation}
\nabla^2L_o(J(t), Y) = \lambda(t) \langle J(t), Y\rangle \quad\mathrm{and}\quad \nabla^2L_{\bar o}(\bar J(t), \bar Y) = \lambda(t) \langle \bar J(t), \bar Y\rangle
\end{equation}
for every $Y\in T_{\gamma(t)}M$ and $\bar Y \in T_{\bar\gamma(t)}\bar M$.
To compute the Hessians we can look in geodesic coordinates about $o$ and $\bar o$ to extend the velocity vector fields of the geodesics to the radial fields $T$ and $\bar T$ and  the Jacobi fields to $J$ and $\bar J$ which commute with $T$ and $\bar T$ respectively. Extend $Y$ and $\bar Y$, then a short computation shows that
\begin{equation}
\nabla^2L_o(J, Y) =  \langle \nabla_TJ, Y\rangle \quad\mathrm{and}\quad \nabla^2L_{\bar o}(\bar J, \bar Y) =  \langle \nabla_{\bar T}\bar J, \bar Y\rangle.
\end{equation}
Thus  along $\gamma$ amd $\bar\gamma$
\begin{equation}
\nabla_T J = \lambda J \quad\mathrm{and}\quad \nabla_{\bar T}\bar J = \lambda \bar J.
\end{equation}
Therefore
$$ \nabla_T \left( \frac 1 {y(t)} J\right) = 0 \quad \mathrm{and}\quad \nabla_{\bar T} \left( \frac 1 {y(t)} \bar J\right) = 0 $$
which shows that 
\begin{equation}
 P(t) = \frac 1 {y(t)} J \quad \mathrm{and}\quad \bar P(t) = \frac 1 {y(t)} \bar J
\end{equation}
 are unit parallel vector fields along $\gamma$ and $\bar\gamma$ \end{proof}
 
 \begin{remark}
If $J$ and $\bar J$ are not perpendicular Jacobi fields in Proposition \ref{p:eqlty}, then we can write
$J(t) = ct T + J_0(t)$ and $\bar J(t) = ct \bar T + \bar J_0(t)$ where $J_0$ and $\bar J_0$ are perpendicular Jacobi Fields vanishing at $0$.  Then by same reasoning in the proposition we conclude that $J_0(t) = y_0(t) P(t)$ and $\bar J_0(t) = y_0(t) \bar P(t)$, where $P$ and $\bar P$ are parallel unit fields and $y_0(t) = |J_0(t) | = |\bar J_0(t)|$.
\end{remark}

\begin{corollary}\label{c:rauch2}
Assume the model surface $(\widetilde M, \tilde o)$ has weaker radial attraction than the complete pointed Riemannian manifold $(M,o)$, and let $\Phi :T_{\tilde o} \widetilde M\rightarrow T_oM$ be a linear isometric inclusion.  Let $X : [a,b] \rightarrow T_{\tilde o}\widetilde M$ be a smooth curve such that $|X(s)| < \ell$ and the geodesics $\gamma_s(t) = exp_o(t \Phi X(s))$ are cut point free for $ 0\leq t \leq 1$ and for all $s \in [a,b]$.  Then
$$
 \mathrm{Length}(\exp_o\circ\Phi\circ X) \leq \mathrm{Length}(\exp_{\tilde o}\circ X).
$$
\end{corollary}
\begin{proof}
Set $\tilde \gamma_s(t) = \exp_{\tilde o}(t X(s))$.  Then $\gamma_s$ and $\tilde\gamma_s$ are variations through geodesics, and hence their transverse fields are Jacobi fields $J_s$ and $\tilde J_s$ along $\gamma_s$ and $\tilde\gamma_s$ respectively which satisfy $J_s(0)=0$, $\tilde J_s(0) = 0$, $J_s^\prime(0) = \Phi (X^\prime(s))$ and $\tilde J_s^\prime(0) =  X^\prime(s)$.  Since $\Phi$ is a linear isometric inclusion, the conditions on the initial conditions of $J_s$ and $\tilde J_s$ in Theorem \ref{t:rauch} are satisfied.  Thus
$$  | J_s(1) | \leq |\tilde J_s(1)|$$
for all $s \in [a,b]$.  Therefore on integrating
\begin{equation}
 \mathrm{Length}(\exp_o\circ\Phi\circ X) = \int_a^b |J_s(1)| ds        \leq  \int_a^b |\tilde J_s(1)| ds  = \mathrm{Length}(\exp_{\tilde o}\circ X).
 \end{equation}.
\end{proof}

\begin{remark}\label{r:rauch2}
If equality holds in Corollary \ref{c:rauch2}, then  we have $|J_s(1)| = |\tilde J_s(1)|$ for all $s$.   Then by Theorem \ref{t:rauch},  $ |J_s(t)| = |\tilde J_s(t)|$ for all $s$ and $t$.  We can then argue that the surface $S$  ruled by the geodesics $\gamma_s$ and the surface $\tilde  S$ ruled by the $\tilde \gamma_s$ have isometric interiors. 
\end{remark}

\subsection{Maximal Radius Theorem}
\begin{proposition}\label{p:MaxRad}
 Let $M$ be a complete $n$--dimensional Riemannian manifold, and let $o$ be a fixed point in $M$.
Let $\widetilde M$ be a compact model surface with vertex $\tilde o$ whose metric in a normal polar coordinate system around $\tilde o$ takes the form
$$ dr^2 + y(r)^2 d\theta^2, \quad 0 < r <\ell.$$
Assume that $(\widetilde M, \tilde o)$   has weaker radial attraction than $(M,o)$ and that
 the cut locus of $o$ in $M$ is a single point $q$ whose distance  from $o$ is $\ell$. 
Then $M$ is diffeomorphic to a sphere $S^n$ and its metric in geodesic coordinates about $o$ is given by
$$ dr^2 + y(r)^2 d\theta_{n-1}^2$$
where $y$ is the function defining the metric on $\widetilde M$ in polar coordinates and $d\theta_{n-1}^2$ is the standard Riemannian metric of constant curvature 1 on $S^{n-1}$.
\end{proposition}

\begin{proof}
By hypothesis,  every minimizing geodesic emanating from $o$ has length $\ell$ and ends at the point $q$. Thus $M$ is an Allamigeon--Warner manifold  or a Blashke manifold at $o$ of the type that is homeomorphic to a sphere \cite[Chapter 5]{AB}.  Moreover, one may define a smooth mapping $\Psi$ from the unit sphere $\Sigma_o$ in $T_oM$ to the unit sphere $\Sigma_q$ in $T_qM$ by setting $\Psi(X) = -\gamma_X^\prime(\ell)$ where $\gamma_X$ is the geodesic emanating from $o$ whose initial tangent vector $\gamma_X^\prime(0)$  is equal to $X \in \Sigma_o$. Its differential  $\Psi_\ast : T_X\Sigma_o \rightarrow T_{\Psi(X)}\Sigma_q $ can be calculated as follows: Given $Y \in T_X\Sigma_o$,  regard $Y$  as a vector in $T_oM$ which is perpendicular to $X$.  Let $J$ be the Jacobi field along $\gamma_X$ satisfying the initial conditions $J(0)=0$ and $J^\prime(0) = Y$. Then $J$ is an orthogonal Jacobi field along $\gamma_X$. By hypothesis $J(\ell)=0$. Consequently
 $J^\prime(\ell)$ is perpendicular to $\gamma_X^\prime(\ell)$ and so may be regarded as a tangent vector in $T_{\Psi(X)}\Sigma_q $.  It is now clear that $\Psi_\ast(Y) = - J^\prime(\ell)$. (cf. the proof of \cite[Lemma 5.27]{AB}.)
 
 The hypothesis that $(\widetilde M, \tilde o)$ has weaker radial attraction than $(M,o)$ implies that $|\Psi_\ast(Y)| \leq |Y|$ for all $Y \in T_X\Sigma_o$.  To prove this, it will suffice  to show that
 $|\Psi_\ast(Y)| \leq  1$ when $|Y| =1$.  Let $\bar \gamma : [0, \ell] \rightarrow \widetilde M$ be a meridian of $\widetilde M$ emanating from $\tilde o$, then the Hessian comparison in the hypothesis of Theorem \ref{t:rauch} is satisfied along $\gamma_X$ and $\bar\gamma$ on the interval $[0,\ell]$.  Assume $|Y| =1$ and let $J$ be the Jacobi field along $\gamma_X$ satisfying $J(0)=0$ and $J^\prime(0)=Y$.  Since $y(r)$ is the norm of the corresponding Jacobi field along $\bar\gamma$,  Theorem \ref{t:rauch}  implies that
 $ |J(r) | \leq y(r)$ for all $0 \leq r \leq \ell$.  Set $f(r) = \langle J(r), J(r)\rangle$ and $\bar f (r) = y(r)^2$.
 Then $ \frac{f(r)}{\bar f(r)} \leq 1$ for $ 0<r<\ell$.  Thus by two applications of l'H\^opital's rule,
 \begin{eqnarray}\label{e:limit}
 1& \geq& \lim_{r \to \ell^-} \frac{f(r)}{\bar f(r)} = 
  \lim_{r\to \ell^-}\frac {\langle J(r) , J^\prime(r)\rangle}{y(r)y^\prime(r)} = 
  \lim_{r\to \ell^-}
  \frac {\langle
   J^\prime(r),J^\prime(r)\rangle+\langle J(r),J^{\prime\prime}(r)\rangle} {y^\prime(r)^2 + y(r)y^{\prime\prime}(r)} \cr
    & = & \frac {\langle J^\prime(\ell), J^\prime(\ell)\rangle}{y^\prime(\ell)^2} = |J^\prime(\ell)|^2
 \end{eqnarray}
 because $J(\ell)=0$, $y(\ell)=0$ and $y^\prime(\ell)=-1$.  Therefore $|\Psi_\ast(Y)| = |J^\prime(\ell)| \leq 1$.
 
The next step is to  show that $|\Psi_\ast(Y)| = 1$ whenever $|Y|=1$.  To prove this we proceed as follows:  Let $\Omega$ and $\hat \Omega$ be the volume
$(n-1)$ forms on the respective unit spheres $\Sigma_o$ and $\Sigma_q$.  Then $\Psi^\ast \hat \Omega = \lambda \Omega$ for some non--vanishing real valued function $\lambda$ on $\Sigma_o$. Note that if $Y_1, \dots, Y_{n-1}$ is an orthonormal basis for $T_X\Sigma_o$ then
\begin{eqnarray*}
|\lambda(X)| &=& |\lambda(X) \Omega(Y_1,\dots,Y_{n-1})| = |\hat\Omega(\Psi_\ast(Y_1),\dots,\Psi_\ast(Y_{n-1}))| \cr
&\leq& |\Psi_\ast(Y_1)| \cdots |\Psi_\ast(Y_{n-1})|. 
\end{eqnarray*}
Consequently, $|\lambda(X)| \leq 1$ for all $X \in \Sigma_o$, and the inequality is 
strict at points $X \in \Sigma_0$ where there exists a unit tangent vector $Y$  with $|\Psi_\ast(Y)| < 1$.
It follows from this and the change of variables formula that
\begin{equation}\label{e:vol}
Vol(S^{n-1}) = \left| \int_{\Sigma_o} \Omega\right| \geq \left| \int _{\Sigma_o} \lambda\Omega \right| = \left| \int_{\Sigma_o} \Psi^\ast\hat\Omega\right| = \left|\int_{\Sigma_q}\hat\Omega\right| = Vol(S^{n-1}).
\end{equation}
Consequently, equality holds in (\ref{e:vol}), and  we deduce that  
$|\lambda| =1$ everywhere. Therefore $|\Psi_\ast(Y)| = 1$ whenever $|Y|=1$. 
Hence, by (\ref{e:limit}) and Theorem \ref{t:rauch}, 
 for every such $Y$, the Jacobi field $J$ with $J(0)=0$ and $J^\prime(0)=Y$ satisfies
$|J(r)| = y(r)$ for all $ 0\leq r \leq \ell$. By Proposition \ref{p:eqlty} 
the metric on $M$ takes the form 
$$ dr^2 + y(r)^2 d\theta_{n-1}^2$$
in polar coordinates about $o$.  Moreover $M$ must be diffeomorphic to $S^n$ as the map $\Psi$ is an isometry.
\end{proof}

\subsection{ Proof of Theorem \ref{t:MRT} }

Suppose that every geodesic triangle $\triangle opq$ has an Alexandrov triangle $\triangle \tilde  o \tilde p \tilde q$ in $\widetilde M$ and that there is a point $q \in M$ with $d(o,q) = \ell$. By Theorem \ref{t:main}, $\widetilde M$ has weaker radial attraction than $M$. 
Thus the distance of every point of $M$ from $o$ is at most $\ell$ by Corollary \ref{c:rauch1}. Let $\tau$ be any minimizing geodesic emanating from $o$ to some point $p$. By hypothesis, the triangle $\triangle opq$ has a comparison triangle $\triangle\tilde o\tilde p\tilde q$ in $\widetilde M$.  Thus
$d(o,p) + d(p,q) = d(\tilde o, \tilde p) + d(\tilde p,\tilde q) = d(\tilde o, \tilde q) =\ell$.
Hence $\tau$ extends to a minimizing geodesic joining $o$ to $q$.  Since $\tau$ was arbitrary, $q$ is the cut point along every geodesic emanating from $o$. Hence the cut locus of $o$ is the single point $q$. Therefore 
the hypothesis of Proposition \ref{p:MaxRad} is satisfied, and the result follows.

\section{Bad Encounters}

In this section assume $(M,o)$ is a pointed complete Riemannian manifold and $(\widetilde M, \tilde o)$ is a model surface of revolution about the point $\tilde o$. We make no further assumptions about the cut loci of points in $\widetilde M$.  For a given point  $p$  in $M$, let $\tilde p$ be the point on the zero meridian of $\widetilde M$, such that  $ d(\tilde o, \tilde p) = d(o,p)$.
Recall from Section 2 the reference map $\widetilde F : \widetilde M ^+ \to \mathbb{R}^2$ defined by  $\widetilde F(\tilde q) = (L_{\tilde p}(\tilde q), L_{\tilde o}(\tilde q))$ and the similarly defined reference map $F : M \to \mathbb{R}^2$.

\begin{definition}\label{d:badencounter}
Let $\sigma : [0,l] \to M$ be a minimizing geodesic in $M$ emanating from $p \in M$.  
We say that $\sigma$ has an \emph{encounter with the cut locus} at $t_0 \in (0,l)$ if
$F(\sigma(t_0)) \in \widetilde F(C(\tilde p)\cap int(\widetilde M^+))$.  Suppose that $\tilde q $ is the unique point in $C(\tilde p) \cap int(\widetilde M^+)$  such that $\widetilde F(\tilde q) = F(\sigma(t_0))$ and $\alpha$ is the arc in $C(\tilde p)$ joining $\tilde q$ to the trunk. The encounter at $t_0$ is a \emph{bad encounter} if for every $\epsilon >0$ there exists $ t^\ast \in (t_0, t_0+\epsilon)$ such that $L_o ( \sigma(t^\ast)) > L_{\tilde o}(\alpha(t^\ast))$.
\end{definition}

\begin{proposition}\label{p:badencounter}
Suppose the minimizing geodesic  $\sigma : [0,l] \to M$ has an encounter with the cut locus at $t_0 \in (0,l)$ and that there exists a $t_1 \in (t_0, l]$ such that
$L_o ( \sigma(t_1)) > L_{\tilde o}( \alpha(t_1))$. Then there is a bad encounter  at some $\bar t \in [t_0, t_1)$. 
\end{proposition}
\begin{proof}
Set $ \bar t= sup\{ t \in [t_0, t_1) : L_o(\sigma(t)) \leq  L_{\tilde o}(\alpha(t))\}$.  Then $t_0 \leq \bar t < t_1$ and $L_o ( \sigma(t^\ast)) > L_{\tilde o}(\alpha(t^\ast))$ for all $t^\ast \in (\bar t, t_1)$.
\end{proof}

\begin{proposition}
Suppose the minimizing geodesic $\sigma : [0,l] \to M$ emanating from $p$ has a bad encounter with the cut locus at $t_0$. Then choosing any $t^\ast \in (t_0,t_0+\epsilon)$ as in Definition \ref{d:badencounter}, the triangle $\triangle op\sigma(t^\ast)$ does not satisfy Alexandrov convexity from $o$.
\end{proposition}

\begin{proof}
Set  $\tilde q = \widetilde F^{-1}(F(\sigma(t_0)))$.  Since  $L_o( \sigma(t^\ast)) > L_{\tilde o}( \alpha(t^\ast))$, the minimizing geodesic $\tilde \sigma$  joining $\tilde p$ to $\tilde q^\ast = \widetilde F^{-1}(F(\sigma(t^\ast)))$ passes above $\tilde q$, that is, $ L_{\tilde o}(\tilde\sigma(t_0)) > L_{\tilde o}(\tilde q) = L_o(\sigma(t_0))$, which violates Alexandrov convexity.
\end{proof}

This establishes the necessity of the assumption of no bad encounters in Theorem \ref{t:main}.

\subsection{ Ensuring encounters with the cut locus are not bad}

Let us assume that $\sigma$ is a minimizing geodesic in $M$ joining $p$ to $q$. Suppose that  the reference curve $F \circ \sigma$ is contained in $\widetilde F(\widetilde M^+)$ where as usual $d(\tilde p,\tilde o) = d(p,o)$.  Moreover suppose that $\sigma$ encounters the cut locus at $t_0$ for some $ 0 < t_0 < d(p,q)$.  Set $q_0 = \sigma(t_0)$,  set $\tilde q_0 = \widetilde F^{-1}(F(q_0))$, 
and let $\alpha$ denote the arc in $C(\tilde p)$ connecting $\tilde q_0$ to the trunk of $C(\tilde p)$.

\begin{lemma}\label{l:4}
The encounter at $q_0$ is not bad if $(L_o\circ \sigma)_+^\prime(t_0) <( L_{\tilde o}\circ \alpha)_+^\prime(t_0)$.
\end{lemma}
\begin{proof}
Because $L_o(\sigma(t_0)) = L_{\tilde o}(\alpha(t_0))$, $(L_o\circ \sigma)_+^\prime(t_0) <( L_{\tilde o}\circ \alpha)_+^\prime(t_0)$  implies that there exists an $\epsilon > 0$  such that for all $t_0 < t < t_0+\epsilon$ we have 
$L_o(\sigma(t)) < L_{\tilde o}(\alpha(t))$ which shows the encounter is not bad.
\end{proof}

\begin{remark}
Conversely, it is clear that   $(L_o\circ \sigma)_+^\prime(t_0) \leq( L_{\tilde o}\circ \alpha)_+^\prime(t_0)$ if the encounter is not bad.
\end{remark}

Let $\tilde\sigma^\downarrow$ and $\tilde\sigma^\uparrow$ be the lowermost and uppermost minimizing geodesics in $\widetilde M$   joining $\tilde p$ to the cut point $\tilde q_0 \in C(\tilde p)$.  In particular this means that if $\tilde \sigma$ is any minimizing  geodesic joining $\tilde p$ to $\tilde q_0$, then $L_{\tilde o}(\tilde\sigma^\downarrow(t))  \leq L_{\tilde o}( \tilde\sigma(t)) \leq L_{\tilde o}(\tilde\sigma^\uparrow(t))$ for all $ t \in [0,t_0]$. By Lemma \ref{l:formula} we have the following inequalities:
\begin{equation}\label{e:1}
(L_{\tilde o}\circ \tilde\sigma^\uparrow)^\prime(t_0) \leq ( L_{\tilde o}\circ \alpha)_+^\prime(t_0) \leq (L_{\tilde o}\circ \tilde\sigma^\downarrow)^\prime(t_0)
\end{equation}
where $\alpha$ is the arc joining $\tilde q_0$ to the trunk.
Moreover the inequalities are strict when $\tilde\sigma^\downarrow \neq \tilde\sigma^\uparrow$.

Note that generally $\tilde\sigma^\downarrow \neq \tilde\sigma^\uparrow$, except in the case that $\tilde q_0$ is an endpoint of $C(\tilde p)$, and even then only if there is only one minimizing geodesic joining $\tilde p$ to $\tilde q_0$.

 \begin{lemma}\label{l:6}
Assume that there exists an $\epsilon > 0$ such that $L_o(\sigma(t)) \geq L_{\tilde o}(\tilde\sigma^\uparrow(t))$ for all $ t_0 - \epsilon \leq  t  \leq t_0$. Suppose either
\begin{enumerate}
\item $\tilde\sigma^\downarrow \neq \tilde\sigma^\uparrow$, or
\item $\tilde\sigma^\downarrow = \tilde\sigma^\uparrow$ and $(L_o\circ \sigma)_-^\prime(t_0) < (L_{\tilde o}\circ\tilde\sigma^\uparrow)^\prime(t_0)$.
\end{enumerate}
Then $(L_o\circ \sigma)_+^\prime(t_0) <( L_{\tilde o}\circ \alpha)_+^\prime(t_0)$ and the encounter at $q_0$ is not bad.
\end{lemma}

\begin{proof}
We have the following string of inequalities.
\begin{equation}
(L_o\circ\sigma)_+^\prime(t_0) \leq (L_o\circ\sigma)_-^\prime(t_0) \leq (L_{\tilde o}\circ\tilde\sigma^\uparrow)^\prime(t_0) \leq ( L_{\tilde o}\circ \alpha)_+^\prime(t_0).
\end{equation}
For the first inequality see \cite[Corollary 2.3]{HI}.  The second is a consequence of the assumption, and the third follows from equation (\ref{e:1}).  In case (1) the third inequality is strict because of equation (\ref{e:1}), and in case (2) the second inequality is strict.  Hence $q_0$ is not a bad encounter by Lemma \ref{l:4}.
\end{proof}

\begin{definition}
Say that $\sigma$ \emph{approaches}  $\tilde q_0 \in C(\tilde p)$ \emph{from the far side} of the cut locus if there exists an $\epsilon > 0$ such that $L_o(\sigma(t)) > L_{\tilde o}(\tilde q)$ for all $ t_0 - \epsilon < t < t_0$ whenever $ \tilde q \in C(\tilde p) $ with $ d (\tilde p, \tilde q)= t$ and the arc in the cut locus connecting $\tilde q$ to the trunk  passes through $\tilde q_0$. 
\end{definition} 

\begin{remark}
It is vacuously true that if $ \tilde q_0$ is an endpoint of a branch of $C(\tilde p)$, then $\sigma$ approaches $\tilde q_0$ from the far side of the cut locus. This definition is adapted from the notion of \lq\lq intersecting positively\rq\rq  in \cite{ISU}.
\end{remark}

\begin{lemma}\label{l:farside}
Suppose that $\sigma$ approaches $\tilde q_0$ from the far side of the cut locus, and that for each $ 0< t< t_0$, there exists a corresponding  Alexandrov triangle  for every triangle $\triangle op\sigma(t)$.  Then
$L_o(\sigma(t)) \geq L_{\tilde o}(\tilde\sigma^\uparrow(t))$ for all $ 0 \leq  t  \leq t_0$ where $\tilde\sigma^\uparrow$ is the uppermost geodesic joining $\tilde p$ to $\tilde q_0$. 
\end{lemma}

\begin{proof}
Because $\sigma$ approaches $\tilde q_0$ from the far side, the minimizing geodesics $\tilde\sigma^t$ joining $\tilde p$ to $\widetilde F^{-1}(F(\sigma(t)))$ converge to $\tilde\sigma^\uparrow$ as $t$ approaches $t_0$.  Moreover, $L_o(\sigma(t)) = L_{\tilde o}(\tilde\sigma^t(t))\geq L_{\tilde o}(\tilde\sigma^\uparrow(t))$ for each $0<t<t_0$ because of Alexandrov convexity.
\end{proof}

\section{Triangle Comparison}

Here we will  prove that if the model surface $\widetilde M$ has weaker radial attraction than $M$ and if  every minimizing geodesic emanating from $p$ has no bad encounters with the cut locus of $\tilde p$ in $\widetilde M$, then for every $q \in M$ and geodesic triangle $\triangle opq$ in $M$, there exists a corresponding geodesic triangle $\triangle \tilde o\tilde  p\tilde q$ in $\widetilde M$ satisfying:
\begin{enumerate}
\item  
$ \displaystyle d(o,p) = d(\tilde o,\tilde p),    d(o,q) = d(\tilde o,\tilde q), d(p,q) = d(\tilde p,\tilde q);  $
 
\item 
$ \displaystyle d(\tilde o,\tilde \sigma(t)) \leq d(o,\sigma(t))\enspace  \forall t \in [0, d(p,q)]; $

\item 
$ \displaystyle  \measuredangle \tilde p \leq \measuredangle  p, \enspace\measuredangle \tilde q \leq \measuredangle  q;$
\item $ \displaystyle\measuredangle \tilde o \leq \measuredangle  o;$
\item $ \displaystyle d(\tilde p,\tilde \gamma(s)) \leq d(p,\gamma(s))\enspace  \forall s \in [0, d(o,q)].$
\end{enumerate}

\begin{remark}
The point $\tilde p$ is chosen on the 0--meridian so that $d(o,p) = d(\tilde o,\tilde p)$ is automatic.  The point $\tilde q \in \widetilde M^+$ satisfying (1) exists and is unique if and only if  $F(q)$ lies in the image of $\widetilde F$.
\end{remark}

As in previous sections $d(o,p)$ will be denoted by $r_0$. The proof will be broken down into a sequence of lemmas.

\begin{lemma}\label{l:bdry}
 Given a geodesic triangle $\triangle opq$ in $M$ with  $F(q) \in \widetilde F(\widetilde M)$, suppose the reference point $F(q) = (x,y)$  lies on the boundary, that is, either
$ x+y = r_0 $, $ y-x = - r_0 $ or $ y-x = r_0$. Then there exists a corresponding triangle $\triangle \tilde o\tilde p \tilde q$ satisfying (1), (2), (3), (4) and (5).
\end{lemma}
\begin{proof}
Consider each case separately.  Assume first that $x+y=r_0$.  Then $ \sigma \cdot \gamma^{-1} $ forms a minimizing geodesic joining $p$ to $o$, and $\tilde q$ lies on $\tilde \tau$ joining $\tilde o$ to $\tilde p$.  Hence $\measuredangle \tilde p = 0 \leq \measuredangle p$,  $\measuredangle \tilde o = 0 \leq \measuredangle o$, and  $\measuredangle \tilde q = \pi = \measuredangle q$.
Moreover,  $d(\tilde o, \tilde\sigma(t)) = r_0 - t = d(o,\sigma(t))$ and $d(\tilde p,\tilde\gamma(s)) = r_0 -s = d(p,\gamma(s))$.

Assume next that $y = x-r_0$.  Then $\tau^{-1} \cdot \gamma$ is a minimizing geodesic joining $p$ to $q$ and $\tilde o$ lies on $\tilde \sigma$ joining $\tilde p$ to $\tilde q$.  Hence $\measuredangle \tilde p = 0 \leq \measuredangle p$,  $\measuredangle \tilde o = \pi = \measuredangle o$, and  $\measuredangle \tilde q =  0 \leq \measuredangle q$. Moreover,  by the triangle inequality, $d(\tilde o, \tilde\sigma(t)) = |r_0 - t| \leq d(o,\sigma(t))$ and $d(\tilde p,\tilde\gamma(s)) = r_0 + s = d(p,\gamma(s))$.

Lastly assume $y = x+ r_0$.  Thus $ \tau \cdot \sigma$ is a minimizing geodesic joining $o$ to $q$ and $\tilde p$ lies on  $\tilde \gamma$ joining $\tilde o$ to $\tilde q$. Hence $\measuredangle \tilde p = \pi = \measuredangle p$,  $\measuredangle \tilde o = 0 \leq \measuredangle o$, and  $\measuredangle \tilde q =  0 \leq \measuredangle q$. Moreover,  by the triangle inequality, $d(\tilde o, \tilde\sigma(t)) = r_0 + t = d(o,\sigma(t))$ and $d(\tilde p,\tilde\gamma(s)) =| r_0 -s| \leq d(p,\gamma(s))$.
\end{proof}

\begin{lemma}\label{l:degen}
Given a geodesic triangle $\triangle opq$ in $M$ with  $F(q)=(x,y) \in \widetilde F(\widetilde M)$, suppose  that $F(\sigma(t)) = (x',y')$ such that either $x'+y' = r_0$ or $y-x = y'-x'$ for some $t \in (0,d(p,q))$. Then there exists a corresponding triangle $\triangle \tilde o\tilde p \tilde q$ satisfying (1), (2), (3), (4) and (5).
\end{lemma}
\begin{proof}
We will show that the hypothesis implies that either $ x+y = r_0 $, $ y-x = - r_0 $ or $ y-x = r_0$ and then apply Lemma \ref{l:bdry}.  First suppose $x^\prime + y^\prime = r_0$, that is, $d(p,\sigma(t)) + d(\sigma(t),o) = d(p,o)$. Then $\sigma(t)$ lies on a minimizing geodesic joining $o$ to $p$.  Thus either $q$ lies between $p$ and $o$ on this geodesic, that is, $x+y= r_0$, or $o$ lies on $\sigma$, that is $ x - y = r_0$.  Next suppose  $y-x = y^\prime - x^\prime$, that is, $d(o,q) = d(o,\sigma(t))+d(\sigma(t),q)$.  Then $\sigma(t)$ lies on a minimizing geodesic joining $o$ to $q$. Thus either, $p$ lies between $o$ and $q$ on this geodesic, that is, $y - x = r_0$, or $o$ lies on $\sigma$, that is, $ x = r_0 + y$.
\end{proof}

The two preceding  lemmas treat the degenerate cases. They require  neither the hypothesis of   weaker radial attraction nor that of  no bad encounters. The following lemma is where these hypotheses are employed.

\begin{lemma}\label{l:nobad}
Assume that $\widetilde M$ has weaker radial attraction than $M$.  Given a geodesic triangle $\triangle opq$ in $M$ with  $F(q) \in \widetilde F(\widetilde M)$, suppose that $\sigma$ has no bad encounters with the cut locus of $\tilde p$. Then there exists a triangle $\triangle\tilde o\tilde p\tilde q$ in $\widetilde M$ satisfying (1),  (2), and (3).
\end{lemma}

\begin{proof}
Because $ F(q)  \in \widetilde F(\widetilde M)$,  there exists a unique $\tilde q \in \widetilde M^+$ such that  $\widetilde F(\tilde q) = F(q)   =(x_0,y_0)$.
By Lemma \ref{l:degen} we may assume that  for all  $t \in (0, d(p,q))$,
\begin{equation} \label{e:nondeg}
 x'+y' > r_0 \enspace \mathrm{and}\enspace y_0-x_0 < y' - x'
 \end{equation}
where $F(\sigma(t) )= (x',y')$. 

 Choose  $\tilde \sigma$ to be the unique minimizing geodesic joining $\tilde p$ to $\tilde q$ if $\tilde q \notin C(\tilde p)$ or
 lowermost one if $\tilde q \in C(\tilde p)$.  Let $\tilde \gamma$  be the arc of the meridian joining $\tilde o$ to  $\tilde q$. These choices  determine the triangle $\triangle\tilde o\tilde p\tilde q$. We must prove $ d(\tilde o,\tilde \sigma(t)) \leq d(o,\sigma(t))$ for all  $t \in [0, d(p,q)] $.

Set $f(t) = d(o,\sigma(t))$ for $ 0 \leq t \leq d(p,q)$, and define a family of functions $f_\phi(t)$ for $0 \leq t \leq d(p,q)$ where $0 \leq \phi \leq \measuredangle \tilde p$ as follows: Given $0 \leq \phi \leq \measuredangle \tilde p$, consider the curve $\varsigma_\phi$ previously described  in Definition \ref{d:varsigmaphi}.  This curve initially emanates from $\tilde p$ travelling along the geodesic $\tilde \sigma_\phi$ until, if and when, it meets a cut point at $t= \tau_\phi$, after which  it travels along the arc of  the cut locus joining that point to the trunk.
By construction, the lowermost geodesic $\tilde\sigma$ joining $\tilde p$ to $\tilde q$, satisfies
$\tilde\sigma(t) = \varsigma_{\measuredangle \tilde p}(t)= \sigma_{\measuredangle \tilde p}(t)$
for $0 \leq t \leq d(p,q)$.  Consequently, if $0\leq \phi < \measuredangle\tilde p$, then the curve $\widetilde F(\varsigma_\phi)$ is below $\widetilde F(\tilde\sigma)$ on the interval $[0,d(p,q)]$, so that there exists a parameter value $\check t_\phi \in (0,d(p,q))$ where $\widetilde F(\varsigma_\phi)$ crosses the line $y-x = y_0-x_0 = d(0,q)- d(p,q)$ in the reference space $\widetilde F(\widetilde M)$. This leads to the definition for each $ 0 \leq \phi \leq\measuredangle \tilde p$,
\begin{equation*}
 f_\phi (t) =
\left\{
\begin{array}{ll}
  d(\tilde o,\varsigma_\phi(t)) & \mathrm{if\enspace} 0\leq t \leq \check t_\phi    \\
t+d(o,q)-d(o,p)  &  \mathrm{if\enspace} \check t_\phi \leq t \leq d(p,q).    
\end{array}
\right.
\end{equation*}
Thus $f_\phi(t)$ is continuous in $\phi$ and $t$.
Moreover  $f_{\measuredangle \tilde p}(t) = d(\tilde o, \tilde\sigma(t))$.  Therefore our goal is to prove   $f_{\measuredangle \tilde p}(t) \leq f(t)$ for $ 0 \leq t \leq d(p,q)$. 

By equation (\ref{e:nondeg}), we have that $f_0 (t) < f(t)$ for $0< t < d(p,q)$, 
$f_0^\prime(0) = -1 < f^\prime(0)$,
and $f_\phi(t) < f(t)$ for $\check t_\phi \leq t < d(p,q)$ and $0 \leq \phi <\measuredangle \tilde p$.

Set 
$$ \bar \phi = \sup \{ 0 \leq \phi  \leq \measuredangle \tilde p : f_\phi^\prime(0) < f^\prime(0)\enspace\mathrm{and}\enspace f_\phi(t) < f(t) \enspace\mathrm{for}\enspace 0< t < d(p,q)\}.$$
If $ \bar \phi = \measuredangle \tilde p$, then $f_{\measuredangle\tilde p} \leq f$ which proves property (2). So suppose that $ \bar\phi < \measuredangle\tilde p$. Then
by continuity and compactness,    $f_{\bar\phi} (t)\leq f(t)$ for all $t$ and either $f_{\bar\phi}^\prime(0) = f^\prime(0)$ or there exists a $0 < \bar t<  d(p,q)$ such that $f_{\bar\phi}(\bar t) = f(\bar t)$.  Were such a $\bar t$ to exist, we would have $0< \bar t < \check t_{\bar \phi}$ because of (\ref{e:nondeg}). Thus either
$\varsigma_{\bar\phi}(\bar t)$ is a cut point of $C(\tilde p)$, or it lies on $\tilde\sigma_\phi$.
 In the first case, $\sigma$ has an encounter with $C(\tilde p)$ at the parameter $\bar t$. Because there are no bad encounters, $f(t) \leq f_{\bar\phi}(t)$ for all $ \bar t \leq t \leq \check t_{\bar\phi}$ by Proposition \ref{p:badencounter} which contradicts $f(\check t_{\bar\phi}) > f_{\bar\phi}(\check t_{\bar\phi})$.  
Thus $\bar t$ must be a parameter value along $\tilde\sigma_{\bar\phi}$.  But then applying  Proposition \ref{p:geodcomp}(3)  with $\sigma$ and $\tilde\sigma_{\bar\phi}$, it follows
that $ f(t) = f_{\bar\phi}(t)$ for all $0 \leq t \leq \min( \tau_{\bar\phi}, \check t_{\bar\phi})$.
Depending upon which of  $\tau_{\bar\phi}$ or $ \check t_{\bar\phi}$ is the smaller, this leads either to an encounter with the cut locus at $t= \tau_{\bar\phi}$ which leads to a contradiction as above or to
$f(\check t_{\bar\phi}) = f_{\bar\phi}(\check t_{\bar\phi})$ which is impossible.  
Hence there is no such $\bar t$, and we would have the case $f_{\bar\phi}^\prime(0) = f^\prime(0)$.  
But by Proposition \ref{p:geodcomp}(1), this again leads to $ f(t) = f_{\bar\phi}(t)$ for all $0 \leq t \leq \min( \tau_{\bar\phi}, \check t_{\bar\phi})$ which we just saw was impossible. This completes the proof of property (2).

The angle comparison (3)  follows from Alexandrov convexity from $o$ by \cite[Lemma 4.6]{HI}.
\end{proof}

\begin{lemma}\label{l:5.5}
Given a geodesic triangle $\triangle opq$ in $M$, let $\gamma :[0, d(o,q)] \rightarrow M$ be the side joining $o$ to $q$, and set $q_s=\gamma(s)$.  Consider the family of triangles of the form $\triangle opq_s$ for  $ s \in (0,d(o,q)]$.  Suppose each such $\triangle opq_s$ has a corresponding triangle $\triangle \tilde o\tilde p \tilde q_s$ satisfying (1), (2), and (3).  Then there exists a triangle $\triangle \tilde o\tilde p\tilde q$ in $\widetilde M$ satisfyling (1), (2), (3), (4), and (5). 
\end{lemma}

\begin{remark}\label{r:5.6}
By Lemma \ref{l:nobad}, the hypothesis is satisfied if $\widetilde M$ has weaker radial attraction than $M$, $F(q_s) \in \widetilde F(\widetilde M)$ for all $s$, and every minimizing geodesic emanating from $p$ has no bad encounters with the cut locus of $\tilde p$.
\end{remark}

\begin{proof}
Set $\widehat\theta(s) = \theta(\widetilde F^{-1} (F(q_s)))$ where $\theta$ is the coordinate on $\widetilde M^+$.  It will be enough to show that $\widehat\theta$ is a nonincreasing function. Since $\widehat\theta$ is continuous,  this will be accomplished by showing that that the lower left Dini derivates of $\widehat\theta$ satisfy  $D_- \widehat\theta(s) \leq 0$ for all $0<s\leq d(o,q)$.
We will show the assumption $D_- \widehat\theta(s_0)= 2c_ 0>0$ for some $s_0$ leads to a contradiction.  Observe that $F(q_{s_0})$ does not lie on the upper left, or lower left boundary lines of $\widetilde F(\widetilde M)$ since then $\widehat\theta(s_0) =0$ which is the absolute minimum value of $\widehat\theta$ which would imply that $D_- \widehat\theta(s_0)\leq 0$.  Also $F(q_{s_0})$ does not lie on the  lower right boundary line, $y = x+r_0$ since then, by Lemma \ref{l:rectangle},  we would have $F(s)$  lying on that line for all $s \leq s_0$.   In other words $\widehat\theta(s) = \pi$ for all $ s \leq s_0$ which would imply the $D_-\widehat\theta(s_0)=0$.
This leaves two possibilities,  $F(s_0) = (x_0,y_0)$ satisfies  $-r_0 < y_0-x_0 < r_0$ and either
(i) $r_0 < x_0 + y_0 <2\ell - r_0$ or (ii) $ x_0+y_0 = 2\ell-r_0$ where $\ell \in (0,\infty]$ is the maximum radius of $\widetilde M$.

Now $D_-\widehat\theta(s_0) = 2c_0 >0$ implies that
$$ \frac {\widehat\theta(s_0-h) - \widehat\theta(s_0)}{-h} > c_0 $$
or equivalently
$ \widehat\theta(s_0) - \widehat\theta(s_0-h) > c_0 h$ for all sufficiently small $ h>0$.
Thus in case (i), \cite[Corollary 3.4]{HI} implies that there exists a $C> 0$ such that
\begin{equation}\label{e:5.1}
 L_{\tilde p}( \mu(s_0-h)) - L_{\tilde p}(\tilde q_{s_0-h}) \geq C( \widehat\theta(s_0) - \widehat\theta(s_0-h)) \geq C c_0 h 
 \end{equation} 
for all sufficiently small $h>0$, where $\mu$ is the meridian passing through $\tilde q_{s_0}$.
On the other hand, since $\mu(s_0) = \tilde q_{s_0}$, the left-hand side of  (\ref{e:5.1}) is equal to
\begin{eqnarray*}
& &  L_{\tilde p}( \mu(s_0-h)) - L_{\tilde p}(\mu(s_0)) - ( L_{\tilde p}(\tilde q_{s_0-h})-L_{\tilde p}(\tilde q_{s_0}))\\
&= &  L_{\tilde p}( \mu(s_0-h)) - L_{\tilde p}(\mu(s_0)) - ( L_{ p}( \gamma(s_0-h))-L_{ p}(\gamma(s_0))) 
\end{eqnarray*}
which combines with (\ref{e:5.1}) to obtain
$$
\frac{ L_{\tilde p}( \mu(s_0-h)) - L_{\tilde p}(\mu(s_0))}{-h} - \frac{ L_{ p}( \gamma(s_0-h))-L_{ p}(\gamma(s_0))}{-h} \leq -Cc_0. 
$$
Thus taking the limit as $h\rightarrow 0^+$ gives
$$ 
(L_{\tilde p}\circ \mu)_-^\prime(s_0) - (L_p\circ \gamma)_-^\prime(s_0) \leq - Cc_0 < 0.
$$ which is equivalent to
$$ \cos( \measuredangle \tilde q_{s_0}) - \cos(\measuredangle q_{s_0}) <0$$
contradicting $ \measuredangle \tilde q_{s_0} \leq \measuredangle q_{s_0}$, since cosine is strictly decreasing on $[0,\pi]$.

Case (ii) cannot occur.  If it could, then $\ell$ would have to be finite.   Thus we would have $\widehat\theta(s_0)=\pi$ and for all sufficiently small $h>0$, $\widehat\theta(s_0-h) <\widehat\theta(s_0)=\pi$. Hence $\tilde q_{s_o-h}$ would be in the interior of the reference space, but we have proved in case (i) that $\widehat\theta$ is a nonincreasing function at such points making it impossible for $\widehat\theta$ to increase to the value $\pi$. 
\end{proof}

\begin{lemma}\label{l:5.7}
Assume that $\widetilde M$ has weaker radial attraction than $M$, and that every minimizing  geodesic emanating from $p$ has no bad encounters with the cut locus of $\tilde p$.  Suppose the geodesic triangle $\triangle opq$ has $\measuredangle o < \pi$.  Then $ F(q) \in \widetilde F(\widetilde M)$, and  there exists a triangle $\triangle \tilde o\tilde p\tilde q$ in $\widetilde M$ satisfyling (1), (2), (3), (4), and (5).  
\end{lemma}

\begin{proof}
Let $\gamma$ be the side joining $o$ to $q$.  By  Lemma \ref{l:5.5} it suffices to prove that $F(\gamma(s)) =F( q_s) \in \widetilde F(\widetilde M)$ for all $s \in [0,d(o,q)]$.  If this is not true, let $s_0 = \inf \{ s : F(q_s) \notin 
\widetilde F(\widetilde M)\}$. By continuity of $F$, $F(q_s) \in \widetilde F(\widetilde M)$ for $0 \leq s \leq s_0$.  By Remark \ref{r:5.6} we may apply  Lemma \ref{l:5.5} to $\triangle opq_{s_0}$ to deduce that $\measuredangle \tilde o \leq \measuredangle o < \pi$.  It follows that $ F(q_{s_0})$ is in the interior of $\widetilde F(\widetilde M)$.  By continuity of $F$, there is an $\epsilon > 0$ such that $F(q_s)$ is in the interior of $\widetilde F(\widetilde M)$ for $| s-s_o| < \epsilon$ which contradicts the choice of $s_0$. 
\end{proof}

\begin{lemma}
Assume that $\widetilde M$ has weaker radial attraction than $M$, and that every minimizing  geodesic emanating from $p$ has no bad encounters with the cut locus of $\tilde p$.   Then $F(M) \subset \widetilde F(\widetilde M)$.
\end{lemma}

\begin{proof}
The set of $q \in M$ such that there is a geodesic triangle of the form $\triangle opq$ with $ \measuredangle o < \pi$ is dense in $M$.  By Lemma \ref{l:5.7} for all such $q$, $F(q) \in \widetilde F(\widetilde M)$.  Thus by continuity of $F$ and the fact that $\widetilde F(\widetilde M)$ is closed, $F(M) \subset \widetilde F(\widetilde M)$.
\end{proof}

This completes the proof of Theorem \ref{t:main} except for the convexity condition about $\tau$. But,  by symmetry on interchanging the roles of $p$ and $q$, it follows from the one about $\gamma$.

\section{Examples}

\subsection{The $\lambda$--spheres $\widetilde M_\lambda$}
Faridi and Schucking \cite{FS} studied a  one parameter family of rotationally symmetric Riemannian metrics on the two dimensional sphere that, in geodesic polar coordinates $(r, \theta)$, take the form
 \begin{equation}\label{e:lambda}
ds_\lambda^2 = dr^2 + \frac{\sin^2(r)}{1+ \lambda \sin^2(r)} d\theta^2.
\end{equation}
  For  $\lambda > -1$, let  $\widetilde M_\lambda$  denote the surface with the metric (\ref{e:lambda}).   
 In particular $\widetilde M_0$ is the 2--sphere of constant curvature 1.   The $\widetilde M_\lambda$ make convenient model surfaces because 
their geodesics have explicit formulas in terms of elementary functions.  

According to \cite{FS},  if  $\sigma(t) = (r(t), \theta(t))$ is the unit speed 
geodesic  in $\widetilde M_\lambda$  starting at $\sigma(0) = \tilde p =(r_0,0)$ which is initially  perpendicular to the meridian, then
 \begin{eqnarray}\label{e:geodesic}
r(t) &=&\arccos(\cos(r_0) \cos(t \varphi_0)) \cr
\theta(t) &=&  t\varphi_0\lambda \sin(r_0)  +  \arctan\left(\frac{ \tan(t\varphi_0)}{\sin(r_0)}          \right) 
\end{eqnarray}
 where $\varphi_0 = \frac 1 {\sqrt{ 1 + \lambda \sin^2(r_0)}} $ and the branch of the inverse tangent is chosen appropriately.  
 More generally, one can verify that if  the initial condition satisfies $r(0) = r_0\in (0,\pi)$ and $r^\prime(0)=\dot r_0 \in [-1,1]$, then
 \begin{equation}\label{e:geodesic2}
r(t) = \arccos\left(\cos(r_0) \cos(t \varphi ) -\dot r_0 \sin(r_0)\frac{\sin(t \varphi)}{\varphi} \right)
\end{equation}
where $\varphi = \sqrt\frac{1+\dot r_0^2\lambda\sin^2(r_0)}{1+\lambda\sin^2(r_0)} $.
Also by \cite{FS} the Gaussian curvature of $\widetilde M_\lambda$ is given by the formula
 \begin{equation}\label{e:gauss}
G_\lambda(r) = \frac{ 1 + 3 \lambda - 2\lambda \sin^2(r)}{ (1 + \lambda\sin^2(r))^2}.
\end{equation}

The cut loci of points for a general class of surfaces of revolution, which include the $\lambda$--spheres for  $ \lambda \geq -\frac 2 3$, are described in  \cite{ST}. 
 If $\lambda > 0$, then $C(\tilde p)$ is an arc in the opposite meridian $\theta = \pi$ containing the antipodal point $(\pi-r_0,\pi)$,  while if $ -\frac 2 3 < \lambda < 0$, then $C(\tilde p)$ is an arc contained in the parallel $r = \pi - r_0$ containing the antipodal point $(\pi-r_0,\pi)$.  The results of \cite{ST} fail to apply when $-1 < \lambda \leq  -\frac 2 3$ because the Gaussian curvature (\ref{e:gauss}) in not a monotone function for $0< r <\frac \pi 2$ whenever $\lambda$ is in this range. Fortunately, using (\ref{e:geodesic2}), it is still true that,  for  all $ -1 < \lambda < 0$,  $C(\tilde p)$  is an arc contained in the parallel $r = \pi - r_0$ containing the antipodal point $(\pi-r_0,\pi)$.
 In this case, 
 the endpoints of $C(\tilde p)$ are found where the geodesic (\ref{e:geodesic}) starting at $\tilde p$ perpendicularly to the meridian first meets the parallel $r = \pi-r_0$.   This means that $r(t) = \pi - r_0$ where  $t\varphi _0= \pi$. Hence solving for $t$, the distance from $\tilde p$ to the endpoints of $C(\tilde p)$ is $t = \pi \sqrt{1 + \lambda \sin^2(r_0)}$.
Therefore, when  $ -1 < \lambda < 0$, the reference space for $\widetilde M_\lambda$ 
 is the rectangle
 $$
 \{ (x,y) : r_0 \leq x+y \leq 2\pi-r_0, -r_0 \leq    y-x \leq r_0 \},
 $$
  and the image of the cut locus of $\tilde p$ in $\widetilde F(\widetilde M_\lambda)$ is the horizontal line segment with $y= \pi-r_0$ and $ \pi\sqrt{1 + \lambda \sin^2(r_0)} \leq x \leq \pi$ as pictured in Figure \ref{f:2}.

\begin{figure}[h]
\begin{center}
\setlength{\unitlength}{1 cm}
\begin{picture}(6,6)(-1,-1)
{\color{gray}
\put(0,0){\line(1,0){5}}
\put(0,0){\line(0,1){5}}}
\thicklines
\put(0,2){\line(1,-1){2}}
\put(0,2){\line(1,1){3}}
\put(2,0){\line(1,1){3}}
\put(3,5){\line(1,-1){2}}
\put(3.5,3){\line(1,0){1.5}}
\put(2, -.5){$r_0$}
\put(-.5,2){$r_0$}
\put(5,-.5){$\pi$}
\put(-.5,5){$\pi$}
\end{picture}
\setlength{\unitlength}{1 cm}
\begin{picture}(6,6)(-1,-1)
{\color{gray}
\put(0,0){\line(1,0){5}}
\put(0,0){\line(0,1){5}}}
\thicklines
\put(0,3){\line(1,-1){3}}
\put(0,3){\line(1,1){2}}
\put(3,0){\line(1,1){2}}
\put(2,5){\line(1,-1){3}}
\put(3.5,2){\line(1,0){1.5}}
\put(3, -.5){$r_0$}
\put(-.5,3){$r_0$}
\put(5,-.5){$\pi$}
\put(-.5,5){$\pi$}
\end{picture}
\caption{$\widetilde F(\widetilde M_\lambda)$  with  cut locus for $-1 < \lambda <0$ for $ r_0 < \frac\pi 2$ (left) and $ r_0 > \frac\pi 2$ (right).}
\label{f:2}
\end{center}
\end{figure}
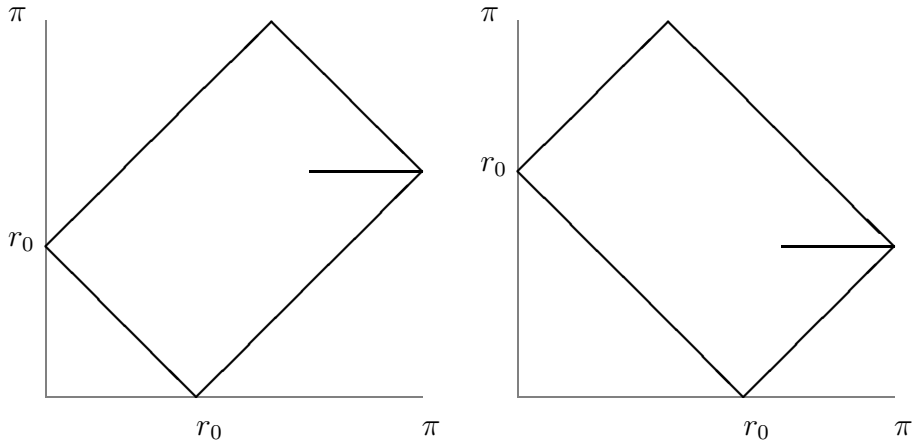

Using equation (\ref{e:lambda})   to calculate the Hessian   of the distance function $L_{\tilde o}$ from the vertex $\tilde o$ in $\widetilde M_\lambda$, one obtains
\begin{equation}\label{e:lHessian}
\nabla^2L_{\tilde o}  = \frac {\cos(r)}{\sin(r) (1 + \lambda\sin^2(r))}(ds_\lambda^2 - dr^2)
\end{equation}
in  geodesic polar coordinates about  $\tilde o$.

\subsection{Comparison of $\mathbb{RP}^n$ with $\lambda$--spheres.}
 
 Let  $\mathbb{RP}^n$ denote  the real projective $n$--space with its metric of constant sectional curvature 1. In a normal coordinate system about  a fixed origin $o \in \mathbb{RP}^n$, its metric takes the form  $ dr^2 + \sin^2(r) d\theta_{n-1}^2$ for $ 0< r < \frac{\pi}{2}$. Here $d\theta_{n-1}^2$ denotes the standard metric on the unit $(n-1)$--dimensional sphere. If $-1 < \lambda < 0$ and $ 0<r <\frac{\pi} {2}$, then $ 0 < 1 + \lambda\sin^2(r) \leq 1$.  Thus
\begin{equation}
\frac {\cos(r)}{\sin(r)} \leq  \frac {\cos(r)}{\sin(r) (1 + \lambda\sin^2(r))}. 
\end{equation}
With (\ref{e:lHessian}), this demonstrates that the Hessian of $L_{\tilde o}$ dominates that of $L_o$.
Therefore  every $\lambda$--sphere  $\widetilde M_\lambda$  with $ -1 < \lambda <0$ has weaker radial attraction than $\mathbb{RP}^n$.

Furthermore,  none of the geodesics in $\mathbb{RP}^n$ have bad encounters with the cut loci of $\widetilde M_\lambda$ when $-1 < \lambda <0$.  The reason is simple.  The distance of any point of $\mathbb{RP}^n$ to the origin $o$ is never greater than $\frac \pi 2$. On the other hand the cut locus of any point  $\tilde p = (r_0,0)$ in $\widetilde M_\lambda$  with $r_0 \leq \frac{\pi}{2}$, being an arc in the co--parallel $ r = \pi-r_0$,   is at a distance $ \pi - r_0 \geq\frac \pi 2$ from $\tilde o$. 
Suppose that $\sigma$ is a geodesic emanating from a point $p$ in  $\mathbb{RP}^n$ with $ d(p,o) =r_0$, then
if $ r_0 < \frac \pi 2$, $\sigma$  cannot encounter any cut points, while if  $r_0 = \frac {\pi} {2}$,  the existence of a bad encounter would lead to the contradiction  $\frac \pi 2 < L_o(\sigma(t^\ast))\leq \frac \pi 2$ for some $t^\ast$. This is clear from Figure \ref{f:3}.  Thus Theorem \ref{t:main} applies to the pair $M = \mathbb{RP}^n$ and $\widetilde M =\widetilde M_\lambda$.
\vskip 1cm
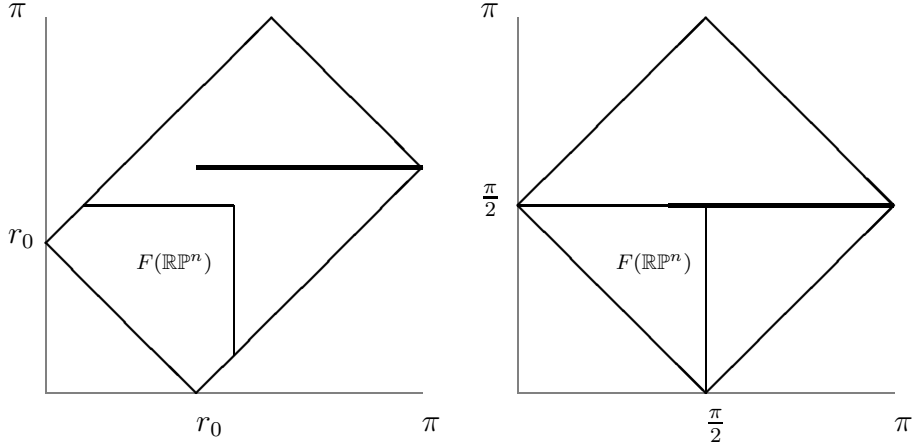
\begin{figure}[h]
\begin{center}
\setlength{\unitlength}{1 cm}
\begin{picture}(6,6)(-1,-1)
{\color{gray}
\put(0,0){\line(1,0){5}}
\put(0,0){\line(0,1){5}}}
\put(.5,2.5){\line(1,0){2}}
\put(2.5,2.5){\line(0,-1){2}}
\thicklines
\put(0,2){\line(1,-1){2}}
\put(0,2){\line(1,1){3}}
\put(2,0){\line(1,1){3}}
\put(3,5){\line(1,-1){2}}
\put(1.2,1.5){$^{F(\mathbb{RP}^n)}$}
\linethickness{.5mm}
\put(5,3){\line(-1,0){3}}
\put(2, -.5){$r_0$}
\put(-.5,2){$r_0$}
\put(5,-.5){$\pi$}
\put(-.5,5){$\pi$}
\end{picture}
\setlength{\unitlength}{1 cm}
\begin{picture}(6,6)(-1,-1)
{\color{gray}
\put(0,0){\line(1,0){5}}
\put(0,0){\line(0,1){5}}}
\put(0,2.5){\line(1,0){2.5}}
\put(2.5,2.5){\line(0,-1){2.5}}
\thicklines
\put(0,2.5){\line(1,-1){2.5}}
\put(0,2.5){\line(1,1){2.5}}
\put(2.5,0){\line(1,1){2.5}}
\put(2.5,5){\line(1,-1){2.5}}
\put(1.3,1.5){$^{F(\mathbb{RP}^n)}$}
\linethickness{.5mm}
\put(5,2.5){\line(-1,0){3}}
\put(2.5, -.5){$\frac \pi 2$}
\put(-.5,2.5){$\frac \pi 2$}
\put(5,-.5){$\pi$}
\put(-.5,5){$\pi$}
\end{picture}
\caption{$F(\mathbb{RP}^n) \subset \widetilde F(\widetilde  M_\lambda)$ for $ r_0 < \frac \pi 2$ (left) and $ r_0 = \frac \pi 2$ (right).}
\label{f:3}
\end{center}
\end{figure}

 \begin{remark}
 
 The preceding example is relevant to  the version of the generalized Toponogov theorem proved in the paper \cite{ISU}:  If the radial curvature of $M$ is bounded from below by that of the model surface $\widetilde M$,  then every geodesic triangle $\triangle  opq$ has a corresponding Alexandrov triangle $\triangle \tilde o \tilde p\tilde q$  provided  one also assumes the condition that
 none of the  local maxima of the distance function $L_o$   restricted to the \lq\lq ellipsoids\rq\rq\space 
$$
E(o,p;r) = \left\{ x \in M : d(o,x) + d(p,x) = r \right\}
$$ 
for $d(o,p) < r $  are mapped to a cut point of $\tilde p$ under the reference map.

However, this condition is stronger than necessary.  Let  $-1 <\lambda < -\frac 3 4$, then there exists $p,x \in \mathbb{RP}^n$ satisfying $d (o,p) = d (o,x) = \frac \pi 2$ and 
$ \pi\sqrt{1+\lambda} < d (p,x) \leq \frac \pi 2$. Hence $x$ maps to a cut point of $\tilde p$ under the reference map  and $x$ is a local maximum  of  $L_o$ restricted to the ellipsoid $E(o,p; d(o,x) +d(p,x))$. Indeed, $x$ is a global maximum of $L_o$ on $\mathbb{RP}^n$.   See Figure \ref{f:3} (right). Thus corresponding Alexandrov triangles exist on account of Theorem \ref{t:main} without the condition about the ellipsoids being satisfied.

 \end{remark}

\subsection{Comparison of spheres of constant curvature $\kappa$ with the $\lambda$--spheres} 
  Another interesting family of examples can be obtained where $M _\kappa$ is a sphere of constant curvature $\kappa$.  
  
  \begin{definition}   
  Let $ 1\leq  \kappa \leq 4$. Define
 \begin{equation*}
\widehat\lambda(\kappa) = \max_{\frac{\pi}{2}< r<\frac{\pi}{\sqrt{\kappa}}}  \frac{\frac{\cot(r) \tan(\sqrt\kappa r)}{\sqrt\kappa} -1} {\sin^2 r}.
\end{equation*}
\end{definition}

\begin{proposition} Let $ 1\leq  \kappa \leq 4$.  If   $\widehat\lambda(\kappa) \leq \lambda \leq 0$, then the $\lambda$--sphere $\widetilde M_\lambda$ has weaker radial attraction that the $n$--sphere $M_\kappa$ of constant sectional curvature $\kappa$.
 \end{proposition}
  
\begin{proof}
By comparing the Hessians of the distance functions,  $\widetilde M_\lambda$ has weaker radial attraction than $M_\kappa$, if and only if
 \begin{equation}\label{e:lambdakappa}
 \frac {\cos(\sqrt \kappa r)}{\sin(\sqrt \kappa r)} \sqrt \kappa \leq \frac {\cos(r)}{\sin(r) (1 + \lambda\sin^2(r))}
\end{equation}
for all $ 0 < r \leq \frac \pi {\sqrt \kappa}$.  
However, assuming $-1 < \lambda \leq 0$ and $ 1 \leq  \kappa\leq 4$, it is automatic that
\begin{equation}
 \frac {\cos(\sqrt \kappa r)}{\sin(\sqrt \kappa r)} \sqrt \kappa \leq \frac {\cos(r)}{\sin(r)} \leq \frac {\cos(r)}{\sin(r) (1 + \lambda\sin^2(r))}
\end{equation}
for $ 0< r < \frac \pi 2$. 
 Consequently, inequality (\ref{e:lambdakappa}) only needs to hold for $ \frac \pi 2 < r < \frac \pi{\sqrt{\kappa}}$ when $-1 < \lambda \leq 0$ and $ 1 \leq  \kappa\leq 4$.  Noting that $\cos(\sqrt \kappa r)$ is negative for $ \frac \pi 2 < r < \frac \pi{\sqrt{\kappa}}$, inequality (\ref{e:lambdakappa}) can be rewritten
\begin{equation}
\lambda \geq \frac{\frac{\cot(r) \tan(\sqrt\kappa r)}{\sqrt\kappa} -1} {\sin^2 r}
\end{equation}
when  $ \frac \pi 2 < r < \frac \pi{\sqrt{\kappa}}$
which concludes the proof.
\end{proof}  
  
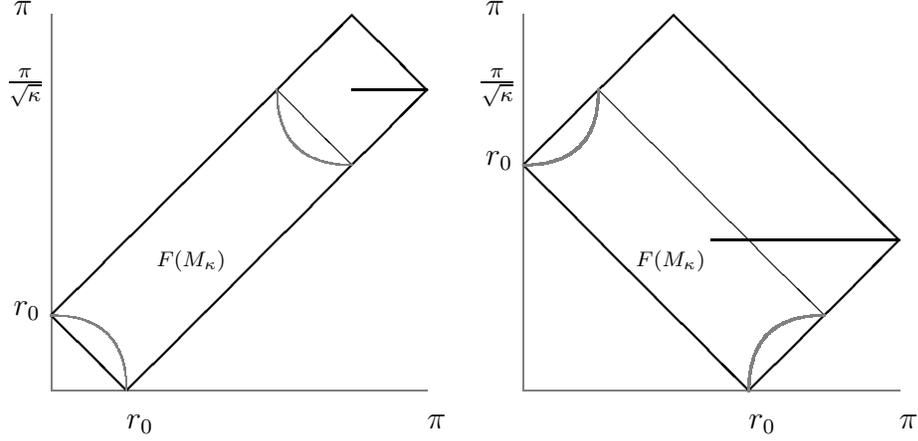
\begin{figure}[h]
\begin{center}
\setlength{\unitlength}{1 cm}
\begin{picture}(6,6)(-1,-1)
{\color{gray}
\put(0,0){\line(1,0){5}}
\put(0,0){\line(0,1){5}}}
\put(3,4){\line(1,-1){1}}
{\thicklines
\put(0,1){\line(1,-1){1}}
\put(0,1){\line(1,1){4}}
\put(1,0){\line(1,1){4}}
\put(4,5){\line(1,-1){1}}
\put(1.4,1.5){$^{ F(M_\kappa)}$}
\put(5,4){\line(-1,0){1}} 
}
\put(1, -.5){$r_0$}
\put(-.5,1){$r_0$}
\put(5,-.5){$\pi$}
\put(-.5,5){$\pi$}
\put(-.6,4){$\frac {\pi} {\sqrt{\kappa}}$}
{\color{gray}
 \qbezier{(0,1), (1,1), (1,0)}
\qbezier{(3,4).(3,3), (4,3)}
}
\end{picture}  
\setlength{\unitlength}{1 cm}
\begin{picture}(6,6)(-1,-1)
{\color{gray}
\put(0,0){\line(1,0){5}}
\put(0,0){\line(0,1){5}}}
\put(1,4){\line(1,-1){3}}
\thicklines
\put(0,3){\line(1,-1){3}}
\put(0,3){\line(1,1){2}}
\put(3,0){\line(1,1){2}}
\put(2,5){\line(1,-1){3}}
\put(1.5,1.5){$^{ F(M_\kappa)}$}
\put(5,2){\line(-1,0){2.5}}
\put(3, -.5){$r_0$}
\put(-.5,3){$r_0$}
\put(5,-.5){$\pi$}
\put(-.5,5){$\pi$}
\put(-.6,4){$\frac {\pi} {\sqrt{\kappa}}$}
{\color{gray}
\qbezier{(3,0),(3,1),(4,1)}
\qbezier{(0,3),(1,3),(1,4)}
}
\end{picture}

\caption{$F(M_\kappa) \subset \widetilde F(\widetilde  M_\lambda)$ with cut locus and the null curves $\mathfrak{s} =0$ represented by overlaying  Figure \ref{f:2a} on Figure \ref{f:2}.  $\widetilde F(C(\tilde p)) \cap F(M_\kappa)$ is empty
for $ r_0 \leq \pi -\frac \pi {\sqrt{\kappa}}$ (left) and  contained in the region where $\frak{s} < 0$ for $\frac \pi{2} < r_0 < \frac \pi {\sqrt{\kappa}}$(right)  .}
\label{f:5a}
\end{center}
\end{figure}

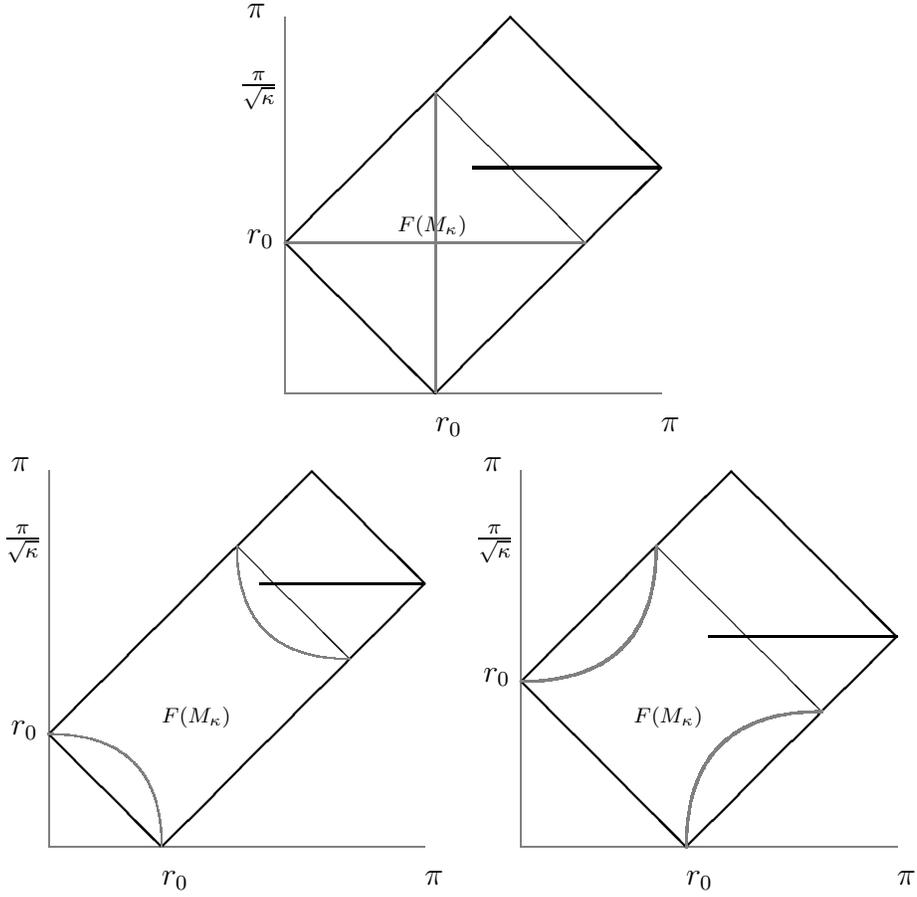
\begin{figure}[h]
\begin{center}

\setlength{\unitlength}{1 cm}
\begin{picture}(6,6)(-1,-1)
{\color{gray}
\put(0,0){\line(1,0){5}} 
\put(0,0){\line(0,1){5}}} 
\put(2,4){\line(1,-1){2}}
\thicklines
\put(0,2){\line(1,-1){2}}
\put(0,2){\line(1,1){3}}
\put(2,0){\line(1,1){3}}
\put(3,5){\line(1,-1){2}}
\put(1.5,2){$^{ F(M_\kappa)}$}
\put(5,3){\line(-1,0){2.5}} 
\put(2, -.5){$r_0$}
\put(-.5,2){$r_0$}
\put(5,-.5){$\pi$}
\put(-.5,5){$\pi$}
\put(-.6,4){$\frac {\pi} {\sqrt{\kappa}}$}
{\color{gray}
\put(0,2){\line(1,0){4}}
\put(2,0){\line(0,1){4}}
}
\end{picture}

\setlength{\unitlength}{1 cm}
\begin{picture}(6,6)(-1,-1)
{\color{gray}
\put(0,0){\line(1,0){5}}
\put(0,0){\line(0,1){5}}}
\put(2.5,4){\line(1,-1){1.5}}
{\thicklines
\put(0,1.5){\line(1,-1){1.5}}
\put(0,1.5){\line(1,1){3.5}}
\put(1.5,0){\line(1,1){3.5}}
\put(3.5,5){\line(1,-1){1.5}}
\put(1.5,1.5){$^{ F(M_\kappa)}$}
\put(5,3.5){\line(-1,0){2.2}} 
}
\put(1.5, -.5){$r_0$}
\put(-.5,1.5){$r_0$}
\put(5,-.5){$\pi$}
\put(-.5,5){$\pi$}
\put(-.6,4){$\frac {\pi} {\sqrt{\kappa}}$}
{\color{gray}
 \qbezier{(0,1.5), (1.5,1.5), (1.5,0)}
\qbezier{(2.5,4).(2.5,2.5), (4,2.5)}
}
\end{picture}  
\setlength{\unitlength}{1 cm}
\begin{picture}(6,6)(-1,-1)
{\color{gray}
\put(0,0){\line(1,0){5}}
\put(0,0){\line(0,1){5}}}
\put(1.8,4){\line(1,-1){2.2}}
\thicklines
\put(0,2.2){\line(1,-1){2.2}}
\put(0,2.2){\line(1,1){2.8}}
\put(2.2,0){\line(1,1){2.8}}
\put(2.8,5){\line(1,-1){2.2}}
\put(1.5,1.5){$^{ F(M_\kappa)}$}
\put(5,2.8){\line(-1,0){2.5}}
\put(2.2, -.5){$r_0$}
\put(-.5,2.2){$r_0$}
\put(5,-.5){$\pi$}
\put(-.5,5){$\pi$}
\put(-.6,4){$\frac {\pi} {\sqrt{\kappa}}$}
{\color{gray}
\qbezier{(2.2,0),(2.2,1.8),(4,1.8)}
\qbezier{(0,2.2),(1.8,2.2),(1.8,4)}
}
\end{picture}
 \caption{$F(M_\kappa) \subset \widetilde F(\widetilde  M_\lambda)$ with cut locus and the null curves $\mathfrak{s} =0$  for $ \pi - \frac \pi{\sqrt{\kappa}} < r_0 < \frac \pi {2\sqrt{\kappa}}$ (lower left), $ r_0 = \frac \pi{2\sqrt{\kappa}}$(upper middle), and  $\frac \pi{2\sqrt{\kappa}} < r_0 < \frac \pi {2}$(lower right). $\widetilde F(C(\tilde p)) \cap F(M_\kappa)$ is always contained in the region where $\frak{s} <0$.}
\label{f:7}
\end{center}
\end{figure}

The function $\widehat\lambda(\kappa)$ is easily computed numerically. See Table \ref{tb:1} for some approximate  values.

\begin{proposition} 
 If $1 \leq \kappa \leq 4$ and
 $ \widehat\lambda(\kappa)  < \lambda <0$, then   minimizing geodesics in $M_\kappa$ do not have any bad encounters with the cut loci in $\widetilde M_\lambda$.  
  \end{proposition}
  \begin{proof}
 Let $r_0 = d(p,o)$ with $ 0 < r_0 < \frac \pi {\sqrt{\kappa}}$.

If $ 0 < r_0 \leq \pi - \frac{\pi} {\sqrt{\kappa}}$, then no minimizing geodesic emanating from $p$ encounters the cut locus because $\frac{\pi} {\sqrt{\kappa}} < \pi - r_0$ shows  $\widetilde F(C(\tilde p))$ is disjoint from  $F(M_\kappa)$. See Figure \ref{f:5a} on the left.  

When  $ \pi - \frac \pi{\sqrt{\kappa}} < r_0 < \frac \pi {\sqrt{\kappa}}$, the cut locus intersects  $F(M_\kappa)$ only if $\pi \sqrt{1 + \lambda \sin^2r_0}\leq \frac {2\pi}{\sqrt{\kappa}} - \pi$, or equivalently, 
$\lambda \sin^2r_0 \leq   \frac {4}{\kappa} -\frac{4}{\sqrt{\kappa}}$.  One may verify  that $\hat\lambda(\kappa) <  \frac {4}{\kappa} -\frac{4}{\sqrt{\kappa}}$ for all $ 1 < \kappa < 4$.  (See Table \ref{tb:1}.) Thus for every  $1 <\kappa<4$ and $\hat\lambda(\kappa) \leq \lambda \leq \frac {4}{\kappa} -\frac{4}{\sqrt{\kappa}}$,  some geodesics emanating from $p$ will  encounter the cut locus as long as $r_0$ is close enough to $\frac \pi 2$. However, by Lemma \ref{l:4},  none of these encounters are bad because the cut locus never extends far enough into $F(M_\kappa)$ to meet the region where  the slope field satisfies $\mathfrak{s}>0$ and in fact remains in the region where $\mathfrak{s} < 0$.
In order to see this there are two cases to consider. 

 In the first case 
suppose $ \pi - \frac \pi{\sqrt{\kappa}} < r_0 < \frac \pi {2}$.  There are three possible configurations pictured in Figure \ref{f:7}. Using (\ref{e:kssf}) to solve the equation $\mathfrak{s}(x,y) =0$  for $x$ with $y = \pi -r_0$,  one finds 
 that the cut locus does not cross the null curve of $\mathfrak{s}$ as long as
\begin{equation}\label{e:nc1}
 \frac {1} {\sqrt{\kappa}} \cos^{-1} \left( \frac{ \cos(\sqrt{\kappa}r_0)}{ \cos(\sqrt{\kappa}(r_0-\pi))}\right) < \pi \sqrt{1 + \lambda \sin^2r_0}.
 \end{equation}
Inequality  (\ref{e:nc1}) may be rearranged into the equivalent inequality
\begin{equation} \label{e:nc2}
\frac {\left( \frac {1}{\pi\sqrt{\kappa}} \cos^{-1}\left(\frac { \cos(\sqrt{\kappa}r_0)}{ \cos(\sqrt{\kappa}(r_0-\pi))} \right)      \right)^2 - 1 }{\sin^2(r_0)}
< \lambda.
\end{equation}
By setting $\hat\mu(\kappa)$ equal to the supremum of the left--hand side of inequality  (\ref{e:nc2}) for $ \pi - \frac \pi{\sqrt{\kappa}} < r_0 < \frac \pi {2}$,  one can verify that $\hat\mu(\kappa) < \hat\lambda(\kappa)$ for all $ 1 < \kappa < 4$. (See Table \ref{tb:1}.)  Since $ \hat\lambda(\kappa) \leq \lambda < 0$, inequality (\ref{e:nc2}) and hence (\ref{e:nc1}) will be satisfied.

In the second case, suppose  $ \frac {\pi}{2} < r_0 < \frac {\pi}{\sqrt{\kappa}}$.  Then every point of $\widetilde F(\widetilde M_\lambda) $ on the horizontal line $ y= \pi - r_0$ is contained in the region with $\mathfrak{s} <0$. See the right  part of  Figure \ref{f:5a}.

\end{proof}

\begin{table}[htbp]
\caption{Values of $\hat\lambda$ and $\hat\mu$ rounded to 5 places.}
\begin{center}
\begin{tabular}{|c||c|c|c|}
\hline
$\sqrt{\kappa}$  & $\hat\mu(\kappa)$    & $\hat\lambda(\kappa) $&   $\frac {4} {\kappa} -\frac {4} {\sqrt{\kappa}}$ \\
\hline
\hline
1.0                 & 0                   &     0             &      0             \\
\hline
1.1                  &   -.74446       & -.50881        &  -.33058        \\
\hline
1.2                  &    -.88571        & -.74151       &      -.55556\\
\hline
1.3                   &    -.94333     &-.85889         &       -.71006\\
\hline
1.4                    &    -.97071     &  -.92212     & -.81633 \\
\hline
1.5                    &     -.98480     & -.95764    &  -.88889\\
\hline
1.6                    &     -.99238      &  -.97803   &   -.93750\\
\hline
1.7                     &     -.99651      &  -.98968    &  -.96886\\
\hline
1.8                    &       -.99869       &-.99607   & -.98765\\
\hline
1.9                    &  -.99972        & -.99914     & -.99723\\
\hline
2.0                  &   -1               &   -1           &     -1     \\
\hline
\end{tabular}
\end{center}
\label{tb:1}
\end{table}%

\section{Topological Applications}

\subsection{A  Sphere Theorem}
 
 In view of Theorem \ref{t:MRT} it seems reasonable to think that if the maximal distance is close enough to $\ell$ then $M$ should be homeomorphic to a sphere.
 Such a result which generalized the Grove--Shiohama Sphere Theorem is proved in \cite{K-O} under the assumption that the model surface is Von Mangoldt and bounds the radial curvature of $M$ from below.  This result can be adapted to our situation. A preliminary lemma is needed to state this result.
 
\begin{lemma} \label{l:top} 
 Let $\widetilde M$ be a compact  model surface with metric $ds^2 = dr^2 + y(r)^2 d\theta^2$, for $ 0 \leq r \leq \ell < \infty$.  Then there exist 
  $0<R<R^\ast<\ell$ such that
 \begin{enumerate}
\item  $y(R) = y(R^\ast)$.
\item $ y(r) > y(R)$ for $ R<r<R^\ast$.
\item $y$ is strictly increasing on $[0,R]$ and strictly decreasing on $[R^\ast,\ell]$.
\item If $0< r_0 <R$ and $\gamma = (r(t),\theta(t))$ is a geodesic starting at $(r_0,0)$ perpendicularly to the meridian, then $\gamma$ eventually meets the parallel $r=R^\ast$ at a parameter value $t_0$ with $ \theta(t_0) > \frac \pi 2$.
\end{enumerate}
\end{lemma}

\begin{proof}
Because $y(0)=0=y(\ell)$,  $y^\prime(0) = 1$, $y^\prime(\ell) = -1$, and $ y(r) > 0$ for $ 0<r<\ell$, it clearly follows that
 if $R$ is sufficiently close to $0$, then there exists an $R^\ast$ close to $\ell$ satisfying (1), (2) and (3).
 Assuming (1), (2) and (3), if   $0<r_0<R$ and $\gamma = (r(t),\theta(t))$ is a geodesic starting at $(r_0,0)$ perpendicularly to the meridian, then $r(t)$ increases until it reaches the value $r_1$ where $y(r_1)=y(r_0)$.  By (1), (2) and (3), $r_1 > R^\ast$.  Thus there exists a $t_0$ with $r(t_0)=R^\ast$.
 To complete the proof it suffices to  show how to pick $R$ small enough so that $ \theta(t_0) > \frac \pi 2$ also holds.
 Observe that if $R$ approaches $0$, then $R^\ast$ will approach $\ell$. 

Let $\sigma$ be the unit speed geodesic through the origin $\tilde o$ of $\widetilde M$ whose trace is the union of the two meridians $\theta = 0$ and $\theta=\pi$.  Assume $\sigma$  is oriented so that $\theta(\sigma(s))$ is equal to $0$ for $s >0$ and to $\pi$ for $s<0$.
Consider the one--parameter family of geodesics $\gamma_s$ such that $\gamma_s(0) = \sigma(s)$ and $\gamma_s^\prime(0)$ is the unit vector  perpendicular to $\sigma$ that points into $\widetilde M^+$. In particular, $\gamma_0$ is the meridian $\theta= \frac \pi 2$.
 The variation vector field of this one--parameter family  restricts to  the Jacobi field $Z$ along $\gamma_0$ satisfying the initial conditions $Z(0) = \sigma^\prime(0)$ and $Z^\prime(0) = 0$.  Thus $Z$ is perpendicular to $\gamma_0$.  Let $P(r)$ denote the parallel  unit vector field perpendicular to $\gamma_0$ along $\gamma_0$  of the form  $P(r) = \frac1{y(r)}\frac{\partial}{\partial \theta}$ for $0<r<\ell$. Then 
 $Z(r) = z(r)P(r)$ 
  for $0\leq r \leq \ell$, 
where $z(r)$ is a function satisfying $z(0)=-1$ and $z^\prime(0)=0$.
Because the first focal point of $\sigma$ along $\gamma_0$ must occur before the first conjugate point of $\sigma(0)$, there exists an $ 0< r_1 <\ell$ such that $z(r_1) =0$. Since the zeros of a non--trivial Jacobi field are simple, it follows that $z(r)$ changes sign at $r=r_1$.  Therefore if $r_2$ is chosen a little bit larger than $r_1$, then for all sufficiently small $s$, the geodesics $\gamma_s$ will cross $\gamma_0$ at some parameter value $r < r_2$.  Thus there exists  a small $R>0$ so that the theta coordinates satisfy $\theta(\gamma_s(r_2)) > \frac \pi 2$ for all $0<s<R$, that is,  $\gamma_s$ must cross the $\theta = \frac \pi 2$ meridian. Therefore (4) is satisfied by choosing $R$ small enough to ensure that $ r_2 < R^\ast$.
 \end{proof}

\begin{proposition}\label{p:top}
Let $\widetilde M$ be a compact model surface, and let $ R$ and R$^\ast$ be chosen as in Lemma \ref{l:top}.  
Suppose that $\widetilde M$ has weaker radial attraction than $M$ and that the geodesics in $M$ have no bad encounters with the cut loci of $\widetilde M$.
If there exists a  point $p \in M$ with distance $d(o,p) > R^\ast$ such that $o$ is a critical point for the distance function from $p$, then $M$ is homeomorphic to a sphere. 
\end{proposition}
 
 \begin{proof}  
 On account of Theorem \ref{t:MRT}, we may assume $L_o$ attains its maximum value $r_{max}$ with 
 $ R^\ast < r _{max} < \ell$.  Because of Lemma \ref{l:top}(3), $y$ is strictly decreasing on $[R^\ast,\ell]$. Hence  for each $R^\ast \leq r \leq \ell$, the geodesic ball $\{ x \in \widetilde M : d(\tilde o,x) \geq r\}$ is strictly convex.  Therefore, the maximum of $L_o$ is attained at a unique point of $M$.  For if it is attained at two points $x_1$ and $x_2$, then let $\triangle\tilde o \tilde x_1 \tilde x_2$ be the Alexandrov triangle corresponding to $\triangle ox_1x_2$.  Thus by the strict convexity of the ball $\{ x \in \widetilde M : d(\tilde o, x) \geq r_{max} \}$ and by Alexandrov convexity, we obtain the contradiction
 $$ r_{max} < d(\tilde o, \tilde\sigma(t)) \leq d(o, \sigma(t)) \leq r_{max}$$
 for $ 0 < t < d(x_1,x_2)$ where $\sigma$ and $\tilde \sigma$ are the corresponding sides of $\triangle ox_1x_2$ and $\triangle\tilde o \tilde x_1 \tilde x_2$ respectively. Therefore the maximum of $L_o$ is attained at a unique point $x_0 \in M$.
 
 Now there are no critical points of $L_o$ in the set $\{ x \in M : d( o,x) \geq R^\ast\}$ other than $x_0$.  If there were, let $q$ be another critical point. Let $\sigma$ be a minimizing geodesic joining $x_0$ to $q$.  Since $x_0$ and $q$ are critical points of $L_o$, we may pick minimizing geodesics $\tau$ from $o$ to $ x_0$ and $\gamma$ from $o$ to $q$ so that $\measuredangle x_0 \leq \frac \pi 2$ and
 $\measuredangle q \leq \frac \pi 2$.
 Consider the geodesic triangle $\triangle ox_0q$ and the corresponding Alexandrov triangle $\triangle \tilde o\tilde x_0\tilde q$.
  Because $ R^\ast \leq d( \tilde o, \tilde q) < d(\tilde o, \tilde x_0)$,  the geodesic $\tilde \sigma$ is contained in the convex disk $\{ x \in \widetilde M : d(\tilde o, x) \geq d(\tilde o, \tilde q)\}$.  Thus its tangent vector at $\tilde q$ points into the disk. On the other hand
 $\tilde\gamma$ is a segment of the meridian starting at $\tilde o$ and thus its tangent vector  points out of the ball at $\tilde q$  and is in fact perpendicular to the boundary of the disk.  Therefore
  $\measuredangle \tilde q > \frac \pi 2$ which contradicts the top angle comparison $\measuredangle \tilde q \leq \measuredangle q$.  Therefore there are no critical points of $L_o$ in the set $\{ x \in \widetilde M : d(\tilde o,x) \geq R^\ast\}$ other than $x_0$.
 
Next we show that there are no  critical points of $L_o$ in the geodesic ball $\{ x \in M : d( o,x) \leq R^\ast\}$ other than $o$.  Here we use the hypothesis  that $o$ is a critical point of $p$.
Suppose there exists a critical point $q$ of $L_o$ with  $0 <d(o,q) \leq R^\ast$.
We construct a geodesic triangle $\triangle opq$ in the following way.
Let $\sigma$ be any minimizing geodesic joining $p$ to $q$.  Since $q$ is a critical point of $L_o$, we may choose a minimizing geodesic $\gamma$ joining $o$ to $q$ so that  $\measuredangle q \leq \frac \pi 2$.
Finally since $o$ is assumed to be a critical point of the distance function from $p$, we may choose $\tau$ joining $o$ to $p$ such that $\measuredangle o \leq \frac \pi 2$.  Let $\triangle\tilde o \tilde p \tilde q$ be the Alexandrov triangle corresponding to $\triangle opq$. The geodesic $\tilde\sigma$ joining $\tilde p$ to $\tilde q$ makes an angle $\measuredangle \tilde q \leq \measuredangle q \leq \frac \pi 2$ with the meridian $\tilde\gamma$ joining $\tilde o$ to $\tilde q$. 
 Because $\tilde\sigma$ starts at a point in $\{ x \in \widetilde M : d(\tilde o,x) > R^\ast\}$, by Lemma \ref{l:top}(2), $\tilde \sigma$ must enter the set $\{ x \in \widetilde M : d(\tilde o,x) < R\}$ and attain a closest distance $r_0<R$ to $\tilde o$ before reaching $\tilde q$.  Thus $\tilde\sigma$ contains a segment which is perpendicular to a meridian at distance $r_0 < R$ from $\tilde o$ and extends to a point at distance $R^\ast$ from $\tilde o$. 
 Hence by Lemma \ref{l:top}(4), the coordinate $\theta$ along $\tilde \sigma$  increases by more than $\frac {\pi}{2}$ along $\tilde\sigma$.  Therefore
 $\measuredangle \tilde o > \frac \pi 2$ which contradicts the angle comparison at the base, $\measuredangle \tilde o \leq \measuredangle o$. 

 The non-existence of  critical points of $L_o$ in  $\{x \in M : 0 < d(o,x) <r_{max}\}$ implies that $M$ is homeomorphic to a sphere by \cite{MG}.
 \end{proof}
  
\subsection{Topological Ends}

\begin{proposition}\label{p:7.3}
Let $\widetilde M$ be a noncompact model surface with metric  $dr^2 + y(r)^2 d\theta^2$ for $0 < r < \infty$.   Suppose that $\widetilde M$ has weaker radial attraction than $M$ and that minimal geodesics  in $M$ have no bad encounters with cut loci in $\widetilde M$.  If  $\liminf_{r\rightarrow \infty} \frac {y(r)}{r}  < \frac {2}{\pi}$, then $M$ has at most one end.
\end{proposition}
 
 \begin{proof}
 Assume that $M$ has at least two ends.  Then there exists a geodesic line $\sigma : \mathbf{R} \rightarrow M$, that is, $d(\sigma(t), \sigma(s)) = |s-t|$ for all $ s,t \in \mathbf{R}$. We can assume that $\sigma(0)$ is the closest point to $o$, and $ d_0 = d(o,\sigma(0))$. Let $t > 0$ and set $\sigma(-t) = p$ and $\sigma(t) = q$.  Thus $ d(p,q) = 2t$ and  by the triangle inequality we have
 \begin{eqnarray}
t - d_0  & \leq& d(o,p) \leq t + d_0  \cr
t- d_0  & \leq& d(o,q) \leq t + d_0. \cr
\end{eqnarray}
Let $\triangle\tilde o\tilde p\tilde q$ denote the Alexandrov triangle corresponding to $\triangle opq$. Then $2t = d(\tilde p, \tilde q) \leq 2 d_0 + \pi y(t)$ if $ t> d_0$ because $\tilde p$ and $\tilde q$ can both be connected to a point on  the parallel $r=t$ by a meridian segment of length at most $d_0$, and the distance between these two points on the parallel is 
at most half the length of the parallel, that is, $\pi y(t)$. In other words, $\frac 2 \pi \leq \frac {2d_0}{t\pi} + \frac {y(t)}{t}$. This leads to the contradiction:   $\liminf_{t\rightarrow \infty} \frac {y(t)}{t}  \geq \frac {2}{\pi}$.
 \end{proof}  
 
 \begin{remark} 
 Any condition that precludes the existence of a geodesic line in $\widetilde M$, {\it e.g.} positive total curvature, could also be used in Proposition \ref{p:7.3}. See \cite{K-O}.
 \end{remark}

\section{Generic Geodesics}

Let $M$ be a complete $n$--dimensional Riemannian manifold with base point $o$.  Recall that the cut locus $C(o)$ of $o$ is the union of a closed singular subset $\mathcal{S}$, whose Hausdorff $(n-1)$-- dimensional  measure is zero, and a relatively open regular subset $\mathcal{R}$ which is a smooth $(n-1)$--dimensional submanifold of $M$. The elements of $\mathcal{S}$ are either conjugate cut points or  cut points  joined to $o$ by at least three minimizing geodesics, while
the elements of $\mathcal{R}$ are cut points joined to $o$ by exactly two nonconjugate  minimizing geodesics.  At a regular cut point, the tangent plane to $\mathcal{R}$ makes the same angle with both of these geodesics.  

\begin{definition}
A minimizing geodesic $\sigma$ from $p \in M$ to $ q\in M$ is \emph{generic} if the interior of $\sigma$ is transverse to $\mathcal{R}$ and disjoint from $\mathcal{S}$. This allows either endpoint $p$ or $q$ to be in $C(o)$. A geodesic triangle $\triangle opq$ is \emph{generic} if the side $\sigma$ joining $p$ to $q$ is generic.
\end{definition}

\begin{proposition}\label{p:dense}
The generic minimizing geodesics emanating from $p$ are dense.
\end{proposition}
\begin{proof} Let $\Sigma_p \subset T_pM$ denote the $(n-1)$--dimensional sphere of  unit tangent vectors at $p$.
The map $\Upsilon : M \backslash (C(p) \cup \{p\}) \rightarrow \Sigma_p $ defined by
$\Upsilon(q) = \sigma_q^\prime(0)$ where $\sigma_q$ is the unique minimizing geodesic joining $p$ to $q\notin C(p)\cup \{p\}$, is smooth and hence locally  Lipschitz. Thus the Hausdorff $(n-1)$--dimensional measure of $\Upsilon(\mathcal{S})$ is zero, and, by Sard's Theorem, the set of  critical values of the restriction
$\Upsilon|\mathcal{R}$ has $(n-1)$--dimensional measure zero as well. Clearly, if $\Upsilon(q)$ is not a critical value  of $\Upsilon|\mathcal{R}$, then $\sigma_q$ is transverse to $\mathcal{R}$, and if also $\Upsilon(q) \notin \Upsilon(\mathcal{S})$ then $\sigma_q$ is disjoint from $\mathcal{S}$. The set of all such $q$ have $n$--dimensional measure zero in $M$.  Thus the generic minimizing geodesics emanating from $p$ are dense.
\end{proof}

\begin{proposition}\label{p:disjoint}
Let $\sigma: [0, l] \to M$ be a generic minimizing geodesic emanating from $p$.  The two--sided derivative $(L_o\circ \sigma)^\prime(t)$ exists for all $0<t<l$, if and only if the interior of $\sigma$ is disjoint from $C(o)$.
\end{proposition}
\begin{proof}
By \cite[Corollary 2.3]{HI}, the two--sided derivative $(L_o\circ \sigma)^\prime(t)$ exists if and only if every minimizing geodesic from $o$ to $\sigma(t)$ makes the same angle with the tangent vector $\sigma^\prime(t)$.  Thus if $ \sigma(t) \notin C(o)$ then $(L_o\circ \sigma)^\prime(t)$ exists  because there is only one minimizing geodesic from $o$ to $\sigma(t)$. Conversely, if $(L_o\circ \sigma)^\prime(t)$ exists  and $ \sigma(t) \in C(o)$, then, since $\sigma$ is generic, $\sigma(t) \in \mathcal{R}$. Hence there are exactly two minimizing geodesics from $o$ to $\sigma(t)$, and they make the same angle with $\sigma^\prime(t)$. This implies that $\sigma^\prime(t)$ is tangent to $\mathcal{R}$ which contradicts the transversality condition in the definition of generic geodesics. Therefore the existence of $(L_o\circ \sigma)^\prime(t)$ implies $\sigma(t) \notin C(o)$.
\end{proof}

\begin{proposition}
Let $\sigma: [0, l] \to M$ be a generic minimizing geodesic emanating from $p$. Then 
$\lim_{t \to t_0^+} (L_o\circ \sigma)_+^\prime(t) = (L_o\circ\sigma)_+^\prime(t_0)$ for all $t_0 \in (0,l)$.
\end{proposition}

\begin{proof}
 If $\sigma(t_0) \in C(o)$, then $\sigma(t_0)$ is a regular cut point. In case $\sigma(t_0) \in C(o)$, let $\gamma_0$ be the 
 minimizing geodesic joining $o$ to $\sigma(t_0)$ such that the tangent vectors $\sigma^\prime(t_0)$ and $\gamma_0^\prime$ point to opposite sides of $\mathcal{R}$.  Otherwise, let $\gamma_o$ be the unique minimizing geodesic joining $o$ to $\sigma(0)$. Thus the $(L_o\circ \sigma)_+^\prime(t_0)$ is  equal to the cosine of the angle between $\sigma$ and $\gamma_0$.
 Let $t_k$ be a decreasing sequence converging to $t_0$ as $k\to\infty$, and  let $\gamma_k$ be a minimizing geodesic joining $o$ to $\sigma(t_k)$ chosen so that the cosine of the angle between $\sigma$ and $\gamma_k$ is equal to $(L_o\circ\sigma)_+^\prime(t_k)$.  Then clearly, $\gamma_k$ converges to $\gamma_0$ as $k \to \infty$.  Therefore $\lim_{t \to t_0^+} (L_o\circ \sigma)_+^\prime(t) = (L_o\circ\sigma)_+^\prime(t_0)$.
\end{proof}

\begin{theorem}
Let $(M,o)$ be a pointed Riemannian manifold and $(\widetilde M,\tilde o)$ a model surface having weaker radial attraction than $(M,o)$.
Let $p\in M$ and let $\tilde p \in \widetilde M$ be the point on the 0--meridiam with $d(\tilde p, \tilde o) =  d(p,o)$.  Assume that every  minimizing geodesic emanating from $p$ approaches the cut locus of $C(\tilde p)$ from the far side.  Then for every geodesic triangle $\triangle opq$ there exists a corresponding Alexandrov triangle $\triangle \tilde o \tilde p\tilde q$ in $\widetilde M$. 
\end{theorem}

\begin{proof}
We first prove this for generic triangles. Given $\triangle opq$, suppose the side $\sigma$ joining $p$ to $q$ is generic.  By Lemma  \ref{l:nobad} it suffices to show that $\sigma$ has no bad encounters with the cut locus.   Let $t_0$ be the supremum of all $t$ such that $\sigma|[0,t]$ has no bad encounters. Certainly, if $t_0 = d(p,q)$ then $\sigma$ has no bad encounters.   Let us assume $t_0 < d(p,q)$ and deduce a contradiction. It follows that  $\sigma$ must have an encounter with the cut locus at $t_0$, for otherwise would contradict the choice of $t_0$.  Since $\sigma$ approaches $C(\tilde p)$ from the far side, Lemma \ref{l:farside} implies
$$ L_o(\sigma(t)) \geq L_{\tilde o}(\tilde \sigma^\uparrow_0(t)) \mathrm{\enspace for\  all\enspace} 0\leq t \leq t_0$$
where $\tilde \sigma^\uparrow_0$ is the uppermost minimizing geodesic joining $\tilde p$ to $\tilde q_0 = \widetilde F^{-1}(F(\sigma(t_0)))$.   Let $\alpha_0$ be the arc in the cut locus joining $\tilde q_0$ to the trunk.  We claim
\begin{equation}\label{e:e0}
(L_o \circ \sigma)^\prime_+(t_0) < (L_{\tilde o}\circ \alpha_0)^\prime_+(t_0).
\end{equation}
There are two cases to consider.  Let $\tilde\sigma_0^\downarrow$ denote the lowermost minimizing geodesic from $\tilde p$ to $\tilde q_0$. If $\tilde\sigma_0^\uparrow \neq \tilde\sigma_0^\downarrow$, then (\ref{e:e0}) follows from Lemma \ref{l:6}. In case $\tilde\sigma_0^\uparrow = \tilde\sigma_0^\downarrow $, equation (\ref{e:e0}) holds as well.  Because if it didn't, then $\tilde q_0$ would be conjugate to $\tilde p$ along $\tilde\sigma_0= \tilde\sigma_0^\uparrow = \tilde\sigma_0^\downarrow$ and we would have
\begin{equation}
(L_o \circ \sigma)^\prime_+(t_0) = (L_{\tilde o}\circ \alpha_0)^\prime_+(t_0)= (L_{\tilde o} \circ \tilde\sigma_o)_+^\prime(t_0).
\end{equation}\label{e:ek} 
Consequently, by Proposition \ref{p:geodcomp}, $L_o(\sigma(t)) = L_{\tilde o}(\tilde\sigma_0(t))$ for all $ 0 \leq t \leq t_0$. Because $L_{\tilde o}\circ\tilde\sigma_0$ is differentiable, so is $L_o\circ \sigma$.  Thus by Proposition \ref{p:disjoint} the interior of $\sigma|[0, t_0]$ is disjoint from $C(o)$.  Hence by Corollary \ref{c:conjugate} with Remark \ref{r:conjugate}, $\sigma|[0, t_0]$ is not conjugate free and so does not minimize past $t_0$. This contradicts $\sigma$ minimizing to distance $d(p,q)$.  This establishes equation (\ref{e:e0}).

This also proves that the encounter at $t_0$ is not bad by Lemma \ref{l:4}.  We will next show that there exists an $\epsilon >0$ such that there are no bad encounters at $t$ for  $t_0 < t < t_0+\epsilon$ which will contradict the choice of $t_0$ assuming $t_0 < d(p,q)$. If no such $\epsilon$  exists, then there is a decreasing sequence $t_k$, $k=1,2,3, \dots$, such that $\lim_{k\to\infty} t_k = t_0$ and $\sigma$ has a bad encounter at $t_k$. For each $k$ set $ \tilde q_k = \widetilde F^{-1}(F(\sigma(t_k)))$ and let
$\alpha_k$ be the arc in $C(\tilde p)$ joining $\tilde q_k$ to the trunk, and let $\tilde \sigma_k^\uparrow$ and $\tilde \sigma_k^\downarrow$ denote the uppermost and lowermost minimizing geodesics joining $\tilde p$ to $\tilde q_k$. By Lemma \ref{l:formula}, for each $k$,
\begin{equation}\label{e:ek}
(L_{\tilde o}\circ\tilde \sigma_k^\uparrow)_+^\prime(t_k)  \leq (L_{\tilde o}\circ \alpha_k)_+^\prime(t_k) \leq(L_{\tilde o}\circ \tilde \sigma_k^\downarrow)_+^\prime(t_k).
\end{equation}
Clearly, both $\tilde\sigma_k^\uparrow$ and $\tilde\sigma_k^\downarrow$ converge to $\tilde\sigma_0^\downarrow$, and by (\ref{e:ek})
\begin{eqnarray*}
\lim_{k \to \infty} \left((L_{\tilde o}\circ \alpha_k)_+^\prime(t_k) - (L_o\circ \sigma)_+^\prime(t_k)\right) &=&
 (L_{\tilde o}\circ \tilde\sigma_0^\downarrow)_+^\prime(t_0) - (L_o\circ \sigma)_+^\prime(t_0)\\
 &\geq& (L_{\tilde o}\circ \alpha_0)_+^\prime(t_0) - (L_o\circ \sigma)_+^\prime(t_0) \\
 &> &0 .
\end{eqnarray*}
Thus for large enough $k$, $(L_{\tilde o}\circ \alpha_k)_+^\prime(t_k)> (L_o\circ \sigma)_+^\prime(t_k)$ so that by Lemma \ref{l:4} the encounters at $t_k$ are not bad after all. Thus there exists an $\epsilon >0$ such that there are no bad encounters at $t$ for  $t_0 < t < t_0+\epsilon$ contradicting the choice of $t_0$.  Therefore $\sigma$ has no bad encounters with the cut locus.

Now consider the general case.  Suppose that  $\triangle opq$ is a geodesic triangle in $M$.  Let $\sigma$  be the side joining $p$ to $q$.  By Proposition \ref{p:dense}, there exists a sequence of generic triangles $\triangle opq_k$ such that the sides $\sigma_k$ joining $p$ to $q_k$ converge to $\sigma$ and are  generic.  By the first part of the proof there exists a sequence of corresponding Alexandrov triangles $\triangle \tilde o\tilde p\tilde q_k$ in $\widetilde M$.  By choosing a subsequence if necessary, we may assume that the sides  $\tilde\sigma_k$ joining $\tilde p$ to $\tilde q_k$ converge to a geodesic $\tilde\sigma$ joining $\tilde p$ to some point $\tilde q$. Thus
by Alexandrov convexity from the base 
$$ d(o, \sigma(t)) = \lim_{k\to\infty} d(o,\sigma_k(t)) \geq \lim_{k\to\infty} d(\tilde o,\tilde\sigma_k(t))= d(\tilde o,\tilde \sigma(t))$$
for all $0\leq t \leq d(p,q)$.  Thus $\triangle\tilde o \tilde p\tilde q$ is an Alexandrov triangle corresponding to $\triangle opq$.
\end{proof}

\section{The slope field at cut points}

Let $\widetilde M$ be a model surface rotationally symmetric about $\tilde o$. Fix $\tilde p \in \widetilde M$, and let $\tilde q_0 \in C(\tilde p)$.  Let $\alpha$  be the  arc in $C(\tilde p)$ joining $\tilde q_0$ to the trunk.  Let $\bar\alpha$ denote  $\alpha$ reparameterized by arclength $s$ from $\tilde q_0$. For each $s \geq 0$ in the domain of $\bar\alpha$, let $\tilde\sigma_s^\uparrow$ and $\tilde\sigma_s^\downarrow$ respectively denote the uppermost and lowermost minimizing geodesics joining $\tilde p$ to $\bar\alpha(s)$. 
Note that if $\bar\alpha(0) = \tilde q_0$ is an endpoint of $C(\tilde p)$, then it may happen that  $\tilde\sigma_0^\uparrow = \tilde\sigma_0^\downarrow$, but otherwise $\tilde\sigma_s^\uparrow \neq \tilde\sigma_s^\downarrow$ if $s > 0$ since $\bar\alpha(s)$ for $s>0$ is not an endpoint of $C(\tilde p)$.   From the construction we have for every $s_0$ in the domain of $\bar\alpha$ that
\begin{equation}\label{e:A1}
\lim_{s \to s_0^+} \tilde\sigma_s^\uparrow =\tilde\sigma_{s_0}^\uparrow\quad\mathrm{and}\quad 
\lim_{s \to s_0^+} \tilde\sigma_s^\downarrow =\tilde\sigma_{s_0}^\downarrow.
\end{equation}
 (The analogous left--hand limit  does not hold at branch points of the cut locus.)
Define
$\psi^\uparrow (s)$ and $\psi^\downarrow(s)$ to  be the respective angles that the respective tangent vectors to  $\tilde\sigma_s^\uparrow$ and $\tilde\sigma_s^\downarrow$ makes with the meridian $\mu$ through $\bar\alpha(s)$, specifically with $+\mu^\prime$.  Set
\begin{equation}\label{e:A2}
\phi(s) = \frac{ \psi^\uparrow(s) + \psi^\downarrow(s)}{2}.
\end{equation}
These three angles may be equal  for $s=0$, if $\tilde q_0$ is an end point of the cut locus, but if $s > 0$, then one has the strict inequality
\begin{equation}\label{e:A3}
0 < \psi^\downarrow(s) < \phi(s) < \psi^\uparrow(s) < \pi.
\end{equation}
By equations (\ref{e:A1}) and (\ref{e:A2}), $\psi^\downarrow$, $\psi^\uparrow$ and $\phi$ are continuous on the right.
On the open dense set of parameters $s$ where $\bar\alpha(s)$ is a regular cut point, $\bar\alpha(s)$ is smooth and $\phi(s)$ is the angle that  the tangent vector $\bar\alpha^\prime(s)$ makes with the meridian through $\bar\alpha(s)$ because  $\bar\alpha^\prime(s)$ makes the same angle with both $\tilde\sigma_s^{\uparrow}$ and $\tilde\sigma_s^{\downarrow}$ by  \cite{O}.

At parameter values $s$ where $\bar\alpha(s)$ is a regular cut point, the first variation formula gives:
\begin{equation}\label{e:A4}
 (L_{\tilde o}\circ\bar\alpha)^\prime(s) = \cos \phi(s) 
 \end{equation}
 and
 \begin{eqnarray}\label{e:A5}
 (L_{\tilde p}\circ \bar\alpha)^\prime(s) &=& \cos(\phi(s) - \psi^\downarrow(s)) = \cos(\psi^\uparrow(s) - \phi(s))\cr
 &=& \cos\left( \frac {\psi^\uparrow(s)-\psi^\downarrow(s)}{2}\right) > 0.
\end{eqnarray}
Likewise, with $t = d(\tilde p, \bar\alpha(s))$, 
\begin{equation}\label{e:A6}
(L_{\tilde o}\circ \tilde\sigma_s^\uparrow)^\prime(t) = \cos \psi^\uparrow(s)\quad\mathrm{and}\quad (L_{\tilde o}\circ \tilde\sigma_s^\downarrow)^\prime(t) = \cos \psi^\downarrow(s).
\end{equation}
The curve $\alpha$ is the curve $\bar\alpha$ parameterized by the distance $t$ from the point $\tilde p$. Thus $\frac{dt}{ds} = (L_{\tilde p}\circ\bar\alpha)^\prime(s)$. Therefore
by the chain rule and equations (\ref{e:A4}) and (\ref{e:A5}) we have
\begin{equation}\label{e:A7}
( L_{\tilde o}\circ\alpha)^\prime(t)  = \frac{ (L_{\tilde o}\circ \bar\alpha)^\prime(s)}{ (L_{\tilde p}\circ \bar\alpha)^\prime(s)} = \frac {  \cos \phi(s) }{ \cos\left( \frac {\psi^\uparrow(s)-\psi^\downarrow(s)}{2}\right) }.
\end{equation}

\begin{lemma} 
If  $\bar\alpha(s)$ is a regular cut point and $t = d(\tilde p, \bar\alpha(s))$, then
\begin{equation*}
 (L_{\tilde o}\circ \tilde\sigma_s^\uparrow)^\prime(t) <  (L_{\tilde o}\circ\alpha)^\prime (t) <(L_{\tilde o}\circ \tilde\sigma_s^\downarrow)^\prime(t) .
\end{equation*}
\end{lemma}
\begin{proof}
The following trig identity is derived from the addition formula for the cosine:
\begin{equation}\label{e:A8}
\cos(A+B) = \cos(2B -(B-A)) = \cos(2B) \cos(B-A) + \sin(2B)\sin(B-A).
\end{equation}
Thus, 
\begin{equation}\label{e:A9}
\frac {\cos(A+B)}{\cos(B-A)} = \cos(2B) + \sin(2B) \tan(B-A).
\end{equation}
On setting $ A = \frac {\psi^\downarrow(s)} 2$ and $B = \frac {\psi^\uparrow(s)} 2$,  we have $A+B = \phi(s)$ and $B-A =  \frac {\psi^\uparrow(s)- \psi^\downarrow(s)} 2 $.
Therefore by (\ref{e:A9})
\begin{equation}\label{e:A10}
\frac{\cos(\phi(s))}{ \cos\left( \frac {\psi^\uparrow(s)- \psi^\downarrow(s)} 2\right)} = \cos(\psi^\uparrow(s)) +  \sin( \psi^\uparrow(s)) \tan \left(  \frac {\psi^\uparrow(s)- \psi^\downarrow(s)} 2 \right).
\end{equation}
Clearly the second term on the right is strictly positive since $\psi^\downarrow(s) \neq \psi^\uparrow(s)$ because $\bar\alpha(s)$, being a regular cut point,  is not an end point.  The inequality on the left now follows from (\ref{e:A10}) on account of (\ref{e:A5}), (\ref{e:A6}) and (\ref{e:A7}). The inequality on the right is proved similarly by setting $ A = \frac {\psi^\uparrow(s)} 2$ and $B = \frac {\psi^\downarrow(s)} 2$ in (\ref{e:A9}) which leads to the equation
\begin{equation}\label{e:A10a}
\frac{\cos(\phi(s))}{ \cos\left( \frac {\psi^\uparrow(s)- \psi^\downarrow(s)} 2\right)} = \cos(\psi^\downarrow(s)) -  \sin( \psi^\downarrow(s)) \tan \left(  \frac {\psi^\uparrow(s)- \psi^\downarrow(s)} 2 \right).
\end{equation}
\end{proof}

The next Lemma calculates the value of the slope field $\frak{s}(x,y)$ at $x= t$ and $y = L_{\tilde o}(\alpha(t))$.
  
\begin{lemma}\label{l:formula}
For every $s$ in the domain of $\bar\alpha$, if $t = d(\tilde p, \bar\alpha(s))$, then the right--hand derivative
 \begin{eqnarray*}
(L_{\tilde o}\circ\alpha)_+^\prime (t) &=& (L_{\tilde o}\circ \tilde\sigma_s^\uparrow)^\prime(t)   + \sin( \psi^\uparrow(s)) \tan \left(  \frac {\psi^\uparrow(s)- \psi^\downarrow(s)} 2 \right) \cr
&=&  (L_{\tilde o}\circ \tilde\sigma_s^\downarrow)^\prime(t)   - \sin( \psi^\downarrow(s)) \tan \left(  \frac {\psi^\uparrow(s)- \psi^\downarrow(s)} 2 \right).
\end{eqnarray*}
 In particular
\begin{equation*}
(L_{\tilde o}\circ \tilde\sigma_s^\downarrow)^\prime(t)  \geq (L_{\tilde o}\circ\alpha)_+^\prime (t) \geq (L_{\tilde o}\circ \tilde\sigma_s^\uparrow)^\prime(t) 
\end{equation*}
where both inequalities are strict if $s>0$.
\end{lemma}

\begin{proof}
Since $\bar\alpha$ is parameterized by arclength, it is a Lipschitz 1 curve. Thus $L_{\tilde p}\circ \bar\alpha$ and $L_{\tilde o}\circ \bar\alpha$ are  Lipschitz 1 functions. Hence they are absolutely continuous.  Consequently, the derivatives $(L_{\tilde p}\circ \bar\alpha)^\prime(s)$ and $(L_{\tilde o}\circ \bar\alpha)^\prime(s)$
exist for almost all $s$.  But for almost all $s$, $\bar\alpha(s)$ is a regular cut point, and at regular cut points, 
equations (\ref{e:A4}) and (\ref{e:A5}) hold.  Thus $L_{\tilde p}\circ \bar\alpha$ and $L_{\tilde o}\circ \bar\alpha$ are absolutely continuous functions whose derivatives are equal almost everywhere to functions which are continuous on the right everywhere.  In other words, they are   indefinite Lebesgue integrals of functions which are continuous on the right at every point. Thus their right--hand derivatives  exist everywhere and satisfy
\begin{equation}\label{e:A11}
(L_{\tilde o}\circ \bar\alpha)_+^\prime(s) = \cos \phi(s) 
\end{equation}
and
\begin{equation}\label{e:A12}
 (L_{\tilde p}\circ \bar\alpha)_+^\prime(s) = \cos\left( \frac {\psi^\uparrow(s)-\psi^\downarrow(s)}{2}\right) > 0.
\end{equation}
By the chain rule
\begin{equation}\label{e:A13}
( L_{\tilde o}\circ\alpha)_+^\prime(t)  = \frac{ (L_{\tilde o}\circ \bar\alpha)_+^\prime(s)}{ (L_{\tilde p}\circ \bar\alpha)_+^\prime(s)} .
\end{equation}
The result now follows from equations (\ref{e:A6}), (\ref{e:A10}), (\ref{e:A10a}),  (\ref{e:A11}), (\ref{e:A12}) and (\ref{e:A13}).
\end{proof}

\end{document}